\documentclass[11pt,twoside]{article}

\usepackage{amsmath}
\usepackage{amssymb}
\usepackage{mathrsfs}
\usepackage{amsthm}

\usepackage{latexsym}

\usepackage{amsrefs}

\usepackage{indentfirst}
\usepackage{color}
\usepackage{txfonts}

\usepackage{anysize}

\allowdisplaybreaks

\usepackage[colorlinks=true,
  linkcolor=red,
  citecolor=blue,
  urlcolor=magenta]{hyperref}

\def\R{{\mathbb R}}
\def\rn{{{\R}^n}}
\def\D{\mathcal{D}}
\def\d{\mathrm{d}}
\def\M{\mathcal{M}}
\def\Mit{\mathcal{M}^{\operatorname{it}}}
\def\H{\mathcal{H}}

\def\sgn{\operatorname{sgn}}

\newtheorem{theorem}{Theorem}[section]
\newtheorem{lemma}[theorem]{Lemma}
\newtheorem{corollary}[theorem]{Corollary}
\newtheorem{proposition}[theorem]{Proposition}
\newtheorem{example}[theorem]{Example}
\theoremstyle{definition}
\newtheorem{remark}[theorem]{Remark}
\newtheorem{definition}[theorem]{Definition}

\numberwithin{equation}{section}

\begin{document}
\title{\bf\Large Mixed Bourgain-Morrey spaces and their applications to boundedness of operators
\footnotetext{\hspace{-0.35cm} 2020 {\it
Mathematics Subject Classification}. Primary 46E30; Secondary  42B20, 42B25, 42B35. \endgraf
{\it Key words and phrases}. fractional integral operator, Hardy-Littlewood maximal operator,  iterated maximal operator,  mixed Bourgain-Morrey space, predual,   singular integral operator.
\endgraf
The work is supported by the National Natural Science Foundation of China (Grant No.
12561002),  the Science and Technology Project of Guangxi (Guike AD25069086)  and Hainan Provincial Natural Science Foundation of China (Grant No. 126MS0135).
}}
\date{}
\author{}
\author{Tengfei Bai, Pengfei Guo and Jingshi Xu\footnote{Corresponding author,
E-mail: \texttt{jingshixu@126.com}}}
\maketitle

\vspace{-0.8cm}

\begin{center}
\begin{minipage}{13cm}
{\small {\bf Abstract:}\quad
We introduce  the mixed Bourgain-Morrey spaces and obtain their preduals.
The boundedness of Hardy-Littlewood maximal operator, iterated maximal operator, fractional integral operator, singular integral operator on these spaces is proved.  In addition, we give a description of the dual of mixed Bourgain-Morrey spaces and  conclude the reflexivity of these spaces.
}
\end{minipage}
\end{center}

\vspace{0.2cm}


\section{Introduction}
In 1961, Benedek and Panzone \cite{BP61} introduced Lebesgue spaces with mixed norm. Bagby \cite{B75}
extended the Fefferman-Stein inequality to the mixed Lebesgue spaces.
In \cite{Ig86}, Igari researched the  interpolation theory for  linear operators in  mixed Lebesgue spaces and apply them to Fourier analysis.
Torres  and Ward considered Leibniz's rule, sampling and wavelets of mixed Lebesgue spaces  in \cite{TW15}.
In \cite{AIV19}, Antoni{\'c},  Ivec  and Vojnovi{\'c}  obtained a general framework for dealing with continuity of linear operators on mixed-norm Lebesgue spaces and showed the boundedness of a large class of pseudodifferential operators on these spaces.
The Hardy-Littlewood-Sobolev inequality on mixed-norm Lebesgue spaces was studied by Chen and Sun in \cite{CS22}.

Then  numerous other function spaces with mixed norms are developed.   Besov spaces with mixed norms  were  studied by in   \cite{BIN78, BIN79}. Triebel-Lizorkin spaces with  mixed norms were considered by  Besov et al. in \cite{BIN96}.  
In 1977,  the Lorentz spaces with mixed norms were first introduced by Fernandez  \cite{Fe77}.
An interpolation result on these spaces was got by Milman  \cite{Mi81}. 
Banach function spaces with mixed norms were studied by  Blozinski   in \cite{Bl81}. 
Anisotropic mixed-norm {Hardy} spaces were  explored by Cleanthous et al. in  \cite{CGN17}.
Mixed-norm $\alpha$-modulation spaces were introduced by Cleanthous and Georgiadis in \cite{CG20}. 
The mixed Lebesgue spaces with variable exponents were studied by Ho in \cite{Ho18}. He obtained the boundedness of operators, such as the  Calder\'on-Zygmund operators on product domains, the Littlewood-Paley operators associated with family of disjoint rectangles,  the nontangential maximal function.  We refer the reader to the survey \cite{HY21} for the development of mixed norm spaces.

Recently,  there are new development of mixed norm spaces. For example, 
in \cite{WYY22}, Wu, Yang and Yuan studied interpolation in mixed-norm function spaces, including mixed-norm Lebesgue spaces, mixed-norm Lorentz spaces and mixed-norm Morrey spaces.
In \cite{Ho16},  Ho obtained the boundedness of the strong maximal operator on mixed-norm spaces. 
In \cite{ZX20}, Zhang and Xue introduced the  classes of multiple weights and of multiple fractional weights. They obtained the boundedness of multilinear strong maximal operators and  multilinear fractional strong maximal operators on weighted mixed norm spaces. 
In \cite{Ho21}, Ho established an extrapolation theory to mixed norm spaces.

In \cite{HLYY20}, Huang et al. introduced the anisotropic mixed-norm Hardy spaces  associated with  a general expansive matrix on  $\rn$ and established their radial or non-tangential maximal function characterizations. 
They obtained  characterizations of  anisotropic mixed-norm Hardy spaces by means of atoms, finite atoms, Lusin area functions, Littlewood-Paley  $g-$functions. The duality between anisotropic mixed-norm Hardy space and anisotropic mixed-norm Campanato space is also obtained. 
In \cite{HCY21}, Huang, Chang and Yang  researched the Fourier transform on anisotropic mixed-norm Hardy space. In \cite{HYY21}, Huang, Yang and Yuan  introduced  anisotropic mixed-norm Campanato-type space associated with a general expansive matrix on $\rn$. They proved that the Campanato-type space is the dual space of the anisotropic mixed-norm Hardy space.

More results can be founded in 
\cite{BKPS06, Bl81, Fe77} for mixed Lorentz spaces, \cite{CGN19,GN16, JS08} for mixed Besov spaces and Triebel-Lizorkin spaces, \cite{DGZ24, DZ23, Hu23, Li22} for mixed Hardy spaces.

Now we turn to the Morrey spaces.
In \cite{Mo38}, Morrey introduced the origin Morrey spaces  to research the  boundedness of the elliptic differential operators. We refer the reader to the  monographs \cite{SDH20, SDH202} for the theory of Morrey spaces. The mixed Morrey spaces were first introduced by  Nogayama in \cite{N19}.
He obtained the  boundedness of iterated
maximal operator,  fractional integral operator and singular integral operator on mixed Morrey spaces. 
Nogayama \cite{N192} studied the predual spaces of mixed Morrey spaces and obtained the necessary and sufficient conditions for the boundedness of commutators of fractional integral operators on mixed Morrey spaces.  

Next we recall the Bourgain-Morrey spaces.
In 1991, Bourgain \cite{Bou91} introduced a special case of Bourgain-Morrey spaces to study the Stein-Tomas (Strichartz) estimate.
After then, many authors began to pay attention to the Bourgain-Morrey spaces. For example,
in \cite{M16}, Masaki showed the block spaces which are the  preduals of Bourgain-Morrey spaces. 
In \cite{HNSH23}, Hatano et al. researched the Bourgain-Morrey spaces from the viewpoints of harmonic analysis and functional analysis.    After then, some function spaces extending Bourgain-Morrey spaces were established.

In \cite{ZSTYY23}, Zhao et al.  introduced Besov-Bourgain-Morrey spaces which connect Bourgain-Morrey spaces with amalgam-type spaces.   They established an equivalent norm with an integral expression  and  showed the boundedness on these spaces of the Hardy-Littlewood maximal operator, the fractional integral, and the Calder\'on-Zygmund operator. The preduals, dual spaces and complex interpolations  of these spaces were also obtained.

Immediately after \cite{ZSTYY23},
Hu, et al.  introduced  Triebel-Lizorkin-Bourgain-Morrey spaces which connect Bourgain-Morrey spaces and global Morrey spaces in \cite{HLY23}. 
The embedding relations between Triebel-Lizorkin-Bourgain-Morrey spaces and Besov-Bourgain-Morrey spaces are proved.  
Various fundamental real-variable properties of these spaces are obtained. The sharp boundedness of  the Hardy-Littlewood maximal operator, the Calder\'on-Zygmund operator, and the fractional integral on these spaces was also proved.

Inspired by the generalized grand Morrey spaces and Besov-Bourgain-Morrey spaces, Zhang et al. introduced generalized grand Besov-Bourgain-Morrey spaces in \cite{ZYZ24}. The preduals  and the Gagliardo-Peetre  interpolation theorem,  extrapolation theorem were obtained.  The boundedness of Hardy-Littlewood maximal operator, the fractional integral and the Calder\'on-Zygmund operator on these spaces was also proved. 

Recently, Zhu et al. \cite{ZYY26} introduced Besov-Bourgain-Morrey spaces on the space of homogeneous type satisfying the reverse doubling property. The boundedness of the Hardy-Littlewood maximal operator, fractional integral operators, fractional maximal operators, Calder\'on-Zygmund operators, and commutators on these spaces was obtained.

The first author and the third author of this paper  introduced the weighted Bourgain-Morrey-Besov-Triebel-Lizorkin spaces associated with operators over a  space of homogeneous type in \cite{BX25} and  obtained two sufficient conditions for precompact sets in matrix weighted Bourgain-Morrey spaces in \cite{BX252}. The authors of this paper  introduced the weighted Bourgain-Morrey-Besov-Triebel-Lizorkin type spaces associated with operators over a  space of homogeneous type in \cite{BGX26}.

Motivated by the above literature, we will study the mixed Bourgain-Morrey spaces and their preduals from the viewpoints of harmonic analysis and functional analysis. The paper is organized as follows.
In Section \ref{preliminaries}, we give the definition of  mixed Bourgain-Morrey spaces and recall some lemmas about mixed Lebesgue spaces.
In Section \ref{sec property}, we research the properties of mixed Bourgain-Morrey spaces, such as embedding, dilation, translation, nontriviality, approximation, density.
In Section \ref{sec predual}, we introduce the mixed norm block spaces and show they are the preduals of mixed Bourgain-Morrey spaces.
In Section \ref{property block}, we consider the properties of the block spaces, such as completeness, density, Fatou property, lattice property. We also show the  duality of  mixed Bourgain-Morrey.
In Section \ref{sec operator}, we consider the boundedness of operators on (sequenced valued) mixed Bourgain-Morrey spaces and the (sequenced valued) block spaces. There operators include Hardy-Littlewood maximal operator, iterated maximal operator, fractional integral operator, singular integral operators.

Throughout this paper,  let $c, C$ denote constants that are independent of the main parameters involved but whose value may differ from line to line.
For $A,B>0$, by $A\lesssim B$, we mean that $A\leq CB$ with some positive constant $C$ independent of appropriate quantities. By $ A \approx B$, we mean that $A\lesssim B$ and $B\lesssim A$.

\section{Preliminaries} \label{preliminaries}
First we recall some notations. Let $\mathbb N  =\{1,2,\ldots\} $, $\mathbb{N}_{0}:=\mathbb{N\cup}\{0\}$ and let $\mathbb Z$ be all the integers. 
Let $\chi_{E}$ be the characteristic function of the set $E\subset\mathbb{R}^{n}$. 
Set $E^c : = \mathbb R^n \backslash E$ for a set $E \subset \rn$.
For $ 1 \le  p \le \infty$, let $p'$  be the conjugate exponent of $p$, that is $1/p +1/p' =1 $. 
For $j\in\mathbb{Z}$, $m\in\mathbb{Z}^{n}$, let $Q_{j,m}:=\prod_{i=1}^{n}[2^{-j}m_{i},2^{-j}(m_{i}+1))$.
For a cube $Q$, $\ell(Q)$ stands for the length of cube $Q$. We
denote by $\mathcal{D} : = \{Q_{j,m},  j\in\mathbb{Z}$, $m\in\mathbb{Z}^{n} \}$ the family of all dyadic cubes in $\mathbb{R}^{n}$,
while $\mathcal{D}_{j}$ is the set of all dyadic cubes with $\ell(Q)=2^{-j},j\in\mathbb{Z}$. Let $a Q , a >0$ be the  cube  concentric with $Q$, having  the side length of $a \ell (Q)$.
For any $R>0$, Let $ B(x,R) := \{y\in \rn : |x-y| <R \} $ be an open ball in $\rn $. Let $B_R : = B (0,R)$ for $R>0$.
The function $\sgn $ is defined by $\sgn (f) = f / |f|$ for $f \neq 0$ and $\sgn (0) = 0$.
Let $ L^0 := L^0 (\rn) $  be the set of  all measurable functions on $\rn$.
Let $\mathbb M ^+ : = \mathbb M ^+  (\rn)$ be the cone of all non-negative Lebesgue measurable functions.
We denote by $\operatorname{Sim} (\rn)$ the class of all simple functions 
$
f =\sum_{i=1}^N a_i  \chi_{E_i}, 
$
where $N \in \mathbb N$, $\{ a_i\} \subset \mathbb C$ and $\{ E_i \} \subset \rn  $ are pairwise disjoint.

The letters $\vec p  = (p_1, \ldots,  p_n)$ will denote $n$-tuples of the numbers in $[0,\infty]$ ($n\in \mathbb N$).
The notation $0< \vec p <\infty$ means that $0 < p_i < \infty$ for each $i \in \{1,\ldots,n\}$.	For $a\in \mathbb R$, let \begin{equation*}
	\frac{1}{\vec p }  = \left( \frac{1}{p_1}, \ldots , \frac{1}{p_n} \right),\quad  a \vec p = (ap_1, \ldots, a p_n), \quad  \vec{p}^{\,\prime} = ( p_1 ', \ldots, p_n ' ).
\end{equation*}

Then we  recall the  mixed Lebesgue spaces, which are  introduced by Benedek and Panzone in \cite{BP61}.
\begin{definition}
	Let $\vec p = (p_1, \ldots, p_n)  \in (0,\infty]^n$. Then the mixed Lebesgue spaces $L^{\vec p} : = L^{\vec p} (\rn)$ is the set of all measurable functions $f :\rn \to \mathbb C$ such that
	\begin{equation*}
		\| f \|_{L^{\vec p} } : = \left(\int_\R \cdots  \left( \int_\R  |f(x_1, \ldots, x_n) |^{p_1 }\d x_1  \right)^{  \frac{p_2}{p_1} }   \cdots \d x_n  \right) ^{\frac{1 }{p_n}} <\infty ,
	\end{equation*}
	with appropriate modifications when  $p_j =\infty$.
	We also define $\| f\|_{L^{\vec p} (E) } := \| f \chi_E \|_{L^{\vec p}} $ for a set $E \subset \rn$.

\end{definition}

Note that the order in which the norms are taken is fundamental because in general (\cite[page 302]{BP61})
\begin{equation*}
	\left(\int_\R \left( \int_\R  |f(x_1, x_2) |^{p_1 }\d x_1  \right)^{  \frac{p_2}{p_1} }    \d x_2  \right) ^{\frac{1 }{p_2}} \neq \left(\int_\R \left( \int_\R  |f(x_1, x_2) |^{p_2 }\d x_1  \right)^{  \frac{p_1}{p_2} }    \d x_2  \right) ^{\frac{1 }{p_1}} .
\end{equation*}

In \cite{N19}, Nogayama introduced  mixed Morrey spaces.

\begin{definition}
	Let $\vec p = (p_1, \ldots, p_n)  \in (0,\infty]^n$ and $t\ \in (0,\infty] $ satisfy $\sum_{i=1 } ^n 1/p_i  \ge n/t.$
	The mixed Morrey norm $\| \cdot \|_{ M_{\vec p} ^t   } $  is defined by 
	\begin{equation*}
		\|f\|_{  M_{\vec p} ^t   } : = \sup_{\operatorname{cube} \; Q \subset \rn } |Q|^{ \frac{1}{t}  - \frac{1}{n} \sum_{i=1 } ^n \frac{1}{p_i}  }  \| f\chi_Q \|_{L^{\vec p}  } ,
	\end{equation*}
	where $f \in L^0 $. The mixed Morrey space $ M_{\vec p} ^t  $ is the set of all measurable functions $f$ with finite norm $ \|f\|_{  M_{\vec p} ^t  }$.
\end{definition}

Inspired by the Bourgain-Morrey space and the mixed Morrey space,
we introduce the mixed Bourgain-Morrey space. 
\begin{definition}
	Let $\vec p = (p_1, \ldots, p_n)  \in (0,\infty]^n$ and $t\ \in (0,\infty] $ satisfy
	$\sum_{i=1 } ^n 1/p_i  \ge n/t.$
	Let $ t \le r \le \infty$. Then the mixed Bourgain-Morrey space $ M_{\vec p} ^{t,r}:=  M_{\vec p} ^{t,r}  (\rn)  $ is the set of all measurable functions $f$ such that
	\begin{equation*}
		\|f\|_{  M_{\vec p} ^{t,r}  } : = \left(\sum_{Q \in \D} |Q|^{ \frac{r}{t}  - \frac{r}{n} \sum_{i=1 } ^n \frac{1}{p_i}  }  \| f\chi_Q \|_{L^{\vec p}  } ^r  \right)^{1/r}<\infty.
	\end{equation*}
\end{definition}

\begin{remark}
	(i) If $ \vec p = (p,p,\ldots, p) $, then $  M_{\vec p} ^{t,r}   $  is the Bourgain-Morrey spaces, which can be found in \cite{HNSH23, M16}.
	
	(ii) If $r =\infty$,  then $  M_{\vec p} ^{t,r}   $   becomes the mixed Morrey space.
	
\end{remark}

We list some properties for $L^{\vec p} $, which will be used frequently.
The first property  is the Fatou's property for $L^{\vec p}$; for example, see \cite[Proposition 2.2]{N19}.
\begin{lemma} \label{lem fatou}
	Let $ 0 < \vec p \le \infty$. Let $\{ f_k\}_{k=1}^\infty$ be a 	sequence of non-negative measurable functions on $\rn$. Then 
	$		\left\|  \liminf _{k\to \infty} f_k \right\|_{L^{\vec p}}  \le \liminf _{k\to \infty}
	\left\|   f_k \right\|_{L^{\vec p}} . $
\end{lemma}

\begin{remark}
	Let $f,g \in L^0 $.
	It is easy to verify that the mixed Bourgain-Morrey space have the lattice property,  that is, $|g| \le |f| $ almost everywhere implies that $\| g\|_{ M_{\vec p} ^{t,r}  }  \le\| f\|_{ M_{\vec p} ^{t,r}   }  $. By the dominated converge theorem, $0 \le f_k \uparrow f$ almost everywhere as $k \to \infty$ implies that $ \|f_k\|_{  M_{\vec p} ^{t,r}  }  \uparrow \| f\|_{  M_{\vec p} ^{t,r}   }$ as $k \to \infty$ for $ n / ( \sum_{i=1}^n  1/p_{i})  \le  t \le r < \infty$.
	From \cite[Theorem 2]{CFMN21}, we obtain $M_{\vec p} ^{t,r}  $ is complete for $ n / ( \sum_{i=1}^n  1/p_{i})  \le  t \le r < \infty$.
	
	When $r=\infty$,  $  M_{\vec p} ^{t,\infty}   $  are the mixed Morrey spaces and  complete.
\end{remark}

The second property is the H\"older inequality, coming from,  for example, \cite[Section 2.1]{N19}.
\begin{lemma} \label{Holder mixed}
	Let $1 <\vec p ,\vec q <\infty$ and define $\vec r$ by $ 1/ \vec p + 1 / \vec q = 1/ \vec r$. If $f \in L^{\vec p} $  and  $g \in L^{\vec q} $, then $fg \in L^{\vec r} $  and
	\begin{equation*}
		\|fg\|_{ L^{\vec r} }  \le 	\|f\|_{ L^{\vec p}  }  	\|g\|_{ L^{\vec q}  } .
	\end{equation*}
\end{lemma}

The third property is duality of $L^{\vec p}$.

\begin{lemma}[Theorem 1, \cite{BP61}] \label{dual mixed Lp}
	
	{\rm (i)}
	Let $1 \le \vec p < \infty$, $1/\vec p + 1 / \vec p \, ^\prime =1 $. $J (f)$ is a continuous linear functional on the normed space $L^{\vec p}$  if and only if it can be represented by 
	\begin{equation*}
		J (f) = \int_\rn h(x) f(x) \d x
	\end{equation*}
	where $h$ is a uniquely determined function of $ L^{\vec p \, ^\prime}$, and $ \| J\| = \|h\|_{ L^{\vec p \, ^\prime}}$.
	
	{\rm (ii)} Let $1 \le \vec p < \infty$. The normed space $L^{\vec p}$ is a Banach space and every sequence converging
	in $L^{\vec p}$ contains a subsequence convergent almost everywhere to the limit 	function.
\end{lemma}

Denote by $\M$  the Hardy-Littlewood maximal function of $f$:
\begin{equation*} 
	\M f (x) = \sup_{B \ni x } \frac{1}{ |B| } \int_B | f (y) | \d y
\end{equation*}
where the sup is taken over all balls $B$ containing $x \in\rn$. For $\eta \in (0,\infty)$, set $\M _\eta (f) = \M (|f|^\eta)^{1/\eta}$.
Finally, we recall the Muckenhoupt's class $A_p$ and it will be used in Section \ref{sec operator}.
\begin{definition}
	Let $1<p<\infty $. We say that a weight $\omega$ belongs to class $A_p := A_p (\rn)$ if
	\begin{equation*}
		[\omega]_{A_p} := \sup_{B\;  \mathrm{ balls \; in  } \;\rn} \Big(   \frac{1}{|B|} \int_B \omega(x) \mathrm {d} x    \Big) ^{1/p} \Big(   \frac{1}{|B|} \int_B \omega(x) ^ { -1/(p-1) } \mathrm {d} x  \Big) ^{ (p-1 ) /p} <\infty.
	\end{equation*}
	We call a weight $\omega$   an $A_1$ weight if
	\begin{equation*}
		\mathcal M (w) (x) \le C w(x)
	\end{equation*}
	for almost all $x \in \rn$ and some $C>0$.
	Set $A_\infty = \cup_{p\ge 1} A_p$.
\end{definition}

\section{Properties of mixed Bourgain-Morrey spaces} \label{sec property}
In this section, we discuss some properties and examples of mixed Bourgain-Morrey spaces.

\subsection{Fundamental properties}
In what follows, the symbol $\hookrightarrow$ stands for continuous embedding. That is, for quasi-Banach spaces $X$ and $Y$, the notation $X\hookrightarrow Y$ means that $X \subset Y$ and that the canonical injection from $X$ to $Y$ is continuous.
\begin{lemma} \label{embed ell r}
	Let $ 0 < \vec p \le \infty $.  Let $ 0 < n / ( \sum_{i=1}^n  1/p_{i}) < t \le r_1 \le  r_2 \le  \infty $.   Then
	\begin{equation*}
		M_{\vec p} ^{t,r_1}    \hookrightarrow  M_{\vec p} ^{t,r_2}   .
	\end{equation*}
\end{lemma}
\begin{proof}
	The inclusion follows immediately from $ \ell^{r_1}  \hookrightarrow  \ell^{r_2}$.
\end{proof}

\begin{proposition} \label{pro embedding}
	Let $ 0 < \vec p  \le \vec s \le \infty $.  Let $ 0< t \le r \le  \infty $ such that $ n/t \le  \sum_{ i=1 }^n 1/ s_i $.   Then
	\begin{equation*}
		M_{\vec s} ^{t,r}    \hookrightarrow  M_{\vec p} ^{t,r}  .
	\end{equation*}
\end{proposition}
\begin{proof}
	Then case $r =\infty$  was proved in  \cite[Proposition 3.2]{N19}. Now let $0<  t \le r <\infty$.
	It suffices to show 
	\begin{equation} \label{s p embed}
		\sum_{Q \in \D} |Q|^{ \frac{r}{t}  - \frac{r}{n} \sum_{i=1 } ^n \frac{1}{p_i}  }  \| f\chi_Q \|_{L^{\vec p}  } ^r   \le 	\sum_{Q \in \D} |Q|^{ \frac{r}{t}  - \frac{r}{n} \sum_{i=1 } ^n \frac{1}{s_i}   }  \| f\chi_Q \|_{L^{\vec s}  } ^r  .
	\end{equation}
	From \cite[Proposition 3.2]{N19}, we have
	\begin{equation*}
		\| f\chi_Q \|_{L^{\vec p}  }  \le \ell (Q) ^{ \left(\sum_{i=1}^n \frac{1}{ p_i}   \right)  -  \left(\sum_{i=1}^n \frac{1}{ s_i}   \right) } \| f\chi_Q \|_{L^{\vec s}  } .
	\end{equation*}
	Put this estimate into the left side of (\ref{s p embed}) and we obtain  the desired result (\ref{s p embed}).
\end{proof}

We also have the dilation property for $ M_{\vec p} ^{t,r}  $. 
\begin{lemma}
	Let $0 < \vec p \le \infty$. Let $0 < n / ( \sum_{i=1}^n  1/p_{i})    \le t < r \le \infty $. Then for $f \in M_{\vec p} ^{t,r}  $,  and $ a >0 $, 
	\begin{equation*}
		\| f(a \cdot )  \|_{ M_{\vec p} ^{t,r}   }  \approx a^{- n/t} \| f  \|_{ M_{\vec p} ^{t,r}   } .
	\end{equation*}
\end{lemma}

\begin{proof}
	The proof is simple and we omit it here. We refer the reader to \cite[(6)]{N19} for the proof of $ M_{\vec p} ^{t,\infty}  $  and to \cite[Lemma 2.4]{HNSH23} for the proof of $ M_{ p} ^{t,r}  $.
\end{proof}

Next, consider the translation invariance of  $ M_{\vec p} ^{t,r}  $.
\begin{lemma}
	Let  $0< \vec p \le \infty  $. Let $0 < n / ( \sum_{i=1}^n  1/p_{i})    \le t < r \le \infty $. Then there exists $C_{n, \vec p, r} >0 $  such that for all $  y \in \rn $  and $f \in  M_{\vec p} ^{t,r}  $,  we have
	\begin{equation*}
		\| f( \cdot -y) \|_{ M_{\vec p} ^{t,r}   }  \le  C_{n, \vec p, r}   \| f \|_{ M_{\vec p} ^{t,r}   } .
	\end{equation*}
\end{lemma}

\begin{proof}
	The proof is similar to \cite[Lemma 2.5]{HNSH23} and we omit it here.
\end{proof}

Using this property and the triangle inequality for $M_{\vec p} ^{t,r}  $, we obtain the following convolution inequality.
\begin{corollary} \label{convolution}
	Let $1 \le \vec p \le \infty $. Let $0 < n / ( \sum_{i=1}^n  1/p_{i})     \le t < r \le \infty $. Then  there exists $C_{n, \vec p, r} >0 $  such that
	\begin{equation*}
		\| g*f \|_{ M_{\vec p} ^{t,r}   }  \le C_{n, \vec p, r}  \|g\|_{L^1}	\| f \|_{ M_{\vec p} ^{t,r}   }  
	\end{equation*}
	for all  $f \in M_{\vec p} ^{t,r}    $  and $g \in L^1$.
\end{corollary}

Before proving Young's inequality for mixed Bourgain-Morrey spaces, we first recall Young's inequality for mixed Lebesgue spaces.

\begin{lemma} [p 319, Theorem 1, \cite{BP61}]  \label{young mixed Lp}
	Let $1 \le \vec p, \vec q, \vec s \le\infty $ satisfy the relation $1/ \vec p + 1/ \vec q = 1 + 1/\vec s $. If $f \in L^{\vec p}, g \in L^{\vec q}$, then  $ f* g \in L^{ \vec s} $  and 
	\begin{equation*}
		\| f*g \|_{ L^{ \vec s}  }  \le \| f \|_{ L^{\vec p} }   \|g\|_{ L^{\vec q} }.
	\end{equation*}
\end{lemma}

\begin{theorem}[Young's inequality]  
	Let parameters $ \vec p , \vec q, \vec s ,  t,t_0 ,t_1 , r , r_0 ,r_1 $ satisfy the following relations:
	
	{\rm (i)} $1 \le \vec p, \vec q, \vec s \le\infty $ and  $1 + 1/\vec p  =  1/ \vec s + 1/ \vec q  $;
	
	{\rm  (ii)}  $ n / ( \sum_{i=1}^n  1/p_{i})    \le t < r \le \infty $,    $ n / (\sum_{i=1}^n 1/  s_i    )  \le t_0 < r_0 < \infty $,    $n / (\sum_{i=1}^n 1/  q_i    )    \le t_1 < r_1 < \infty $;   
	
	{\rm  (iii)}  $1/t + 1 =1/t_0  +1/ t_1  $ and $1/r + 1 =1/ r_0 + 1/ r_1 $.
	
	Then for all $f \in  M_{\vec s} ^{t_0,r_0} $  and $g \in M_{\vec q} ^{t_1,r_1}   $,
	\begin{equation*}
		\| f*g\|_{  M_{\vec p} ^{t,r}  }  \lesssim \|f\|_{ M_{\vec s} ^{t_0,r_0}    } \|g\|_{ M_{\vec q} ^{t_1,r_1}    } .
	\end{equation*}
\end{theorem}
\begin{proof}
	We use the idea from \cite[Theorem 2.7]{HNSH23}. Since the spaces $ M_{\vec s} ^{t_0,r_0}  $ and  $ M_{\vec q} ^{t_1,r_1}   $ are closed when taking the absolute value, we assume $f$ and $g$ are non-negative. This justifies the definition of $f*g$. Fix $v\in \mathbb Z$ and $m \in \mathbb Z^n$.
	By Minkowski's inequality and Young's inequality for the $L^{ \vec p }$-norm (Lemma \ref{young mixed Lp}), we have
	\begin{align*}
		\| (f *g) \chi_{Q_{v , m} } \|_{L^{\vec p}}  & = \left\| \left( \sum_{\tilde m \in \mathbb Z^n} \int_{Q_{v , \tilde m} } f (y) g(\cdot -y) \d y  \right)  \chi_{Q_{v, m} }  \right\|_{L^{\vec p} }  \\
		& \le  \sum_{\tilde m \in \mathbb Z^n} \|  (f \chi_{ Q_{v, \tilde m} })  * (g \chi_{ Q_{v , m} - Q_{v , \tilde m}  })   \|_{ L^{\vec p} }  \\
		& \le  \sum_{\tilde m \in \mathbb Z^n} \|  f \chi_{ Q_{v ,  \tilde m} } \|_{L^{\vec s } } \| g \chi_{ Q_{v , m} - Q_{v  , \tilde m}  }  \|_{ L^{\vec q} }  ,
	\end{align*}
	where $Q_{v, m} - Q_{v  , \tilde m}  := \{  x - \tilde x :x \in Q_{v , m} ,  \tilde x \in Q_{v,  \tilde m}  \} $. Since $ Q_{v, m} - Q_{v ,  \tilde m}  \subset 3 Q_{ v, m-\tilde m} $, by the definition of the parameters $ \vec p , \vec q, \vec s ,  t,t_0 ,t_1 , r , r_0 ,r_1 $, we  obtain
	\begin{align*}
		& \left( \sum_{ m \in \mathbb Z^n}  | Q_{v, m}|^{ \frac{r}{t} - \frac{r}{n} \sum_{i=1}^n \frac{1}{p_i} } \| (f *g) \chi_{Q_{v, m} } \|_{L^{\vec p}} ^r	\right) ^{1/r}\\
		& \le \left( \sum_{ m \in \mathbb Z^n}  | Q_{v, m}|^{ \frac{r}{t} - \frac{r}{n} \sum_{i=1}^n \frac{1}{p_i} } \left( \sum_{\tilde m \in \mathbb Z^n} \|  f \chi_{ Q_{v,  \tilde m} } \|_{L^{\vec s } } \| g \chi_{  3 Q_{ v, m-\tilde m}  }  \|_{ L^{\vec q} }  \right) ^r	\right) ^{1/r}\\
		& \le \left( \sum_{ m \in \mathbb Z^n}  | Q_{v, m}|^{ \frac{r_0}{t_0} - \frac{r_0 }{n} \sum_{i=1}^n \frac{1}{s_i} }  \|  f \chi_{ Q_{v,  m} } \|_{L^{\vec s } } ^{r_0} \right) ^{1/r_0}  
		\left(  \sum_{ m \in \mathbb Z^n}  | Q_{v, m}|^{ \frac{r_1}{t_1} - \frac{r_1 }{n} \sum_{i=1}^n \frac{1}{q_i} }     \| g \chi_{3 Q_{v, m}  }  \|_{ L^{\vec q} }   ^{r_1}	\right) ^{1/r_1} \\
		& \lesssim  \left( \sum_{ m \in \mathbb Z^n}  | Q_{v, m}|^{ \frac{r_0}{t_0} - \frac{r_0 }{n} \sum_{i=1}^n \frac{1}{s_i} }  \|  f \chi_{ Q_{v , m} } \|_{L^{\vec s } } ^{r_0} \right) ^{1/r_0}  
		\left(  \sum_{ m \in \mathbb Z^n}  | Q_{v, m}|^{ \frac{r_1}{t_1} - \frac{r_1 }{n} \sum_{i=1}^n \frac{1}{q_i} }     \| g \chi_{ Q_{v, m}  }  \|_{ L^{\vec q} }   ^{r_1}	\right) ^{1/r_1} 
	\end{align*}
	where we used Young's inequality for the $\ell^r$-norm in the second inequality.  
	Note that for each $j = 0,1 $, $1 <r_j <r  \le   \infty$ and hence $ \ell^{r_j} (\mathbb Z) \hookrightarrow \ell^{r} (\mathbb Z)  \hookrightarrow \ell^{\infty} (\mathbb Z)$. 
	Consequently,
	\begin{align*}
		\| f*g\|_{  M_{\vec p} ^{t,r}  } 
		& \lesssim \left\|\left( \sum_{ m \in \mathbb Z^n}  | Q_{v, m}|^{ \frac{r_0}{t_0} - \frac{r_0 }{n} \sum_{i=1}^n \frac{1}{s_i} }  \|  f \chi_{ Q_{v,  m} } \|_{L^{\vec s } } ^{r_0} \right) ^{1/r_0}  \right\|_{\ell_v^{r} }  
		\left\|	\left(  \sum_{ m \in \mathbb Z^n}  | Q_{v, m}|^{ \frac{r_1}{t_1} - \frac{r_1 }{n} \sum_{i=1}^n \frac{1}{q_i} }     \| g \chi_{ Q_{v, m}  }  \|_{ L^{\vec q} }   ^{r_1}	\right) ^{1/r_1}  \right\|_{\ell_v^{\infty} } \\
		& \le  \left\|\left( \sum_{ m \in \mathbb Z^n}  | Q_{v, m}|^{ \frac{r_0}{t_0} - \frac{r_0 }{n} \sum_{i=1}^n \frac{1}{s_i} }  \|  f \chi_{ Q_{v , m} } \|_{L^{\vec s } } ^{r_0} \right) ^{1/r_0}  \right\|_{\ell_v^{r_0 } } 
		\left\|	\left(  \sum_{ m \in \mathbb Z^n}  | Q_{v, m}|^{ \frac{r_1}{t_1} - \frac{r_1 }{n} \sum_{i=1}^n \frac{1}{q_i} }     \| g \chi_{ Q_{v, m}  }  \|_{ L^{\vec q} }   ^{r_1}	\right) ^{1/r_1}  \right\|_{\ell_v^{r_1} } \\
		& = \|f\|_{ M_{\vec s} ^{t_0,r_0}    } \|g\|_{ M_{\vec q} ^{t_1,r_1}   },
	\end{align*}
	as desired.
\end{proof}

\subsection{Nontriviality}
Next we consider the nontriviality of mixed Bourgain-Morrey spaces.
\begin{lemma}[(2),  \cite{N19}] \label{chi Q mixed Lp}
	Let $Q$  be a cube. Then for $0  < \vec p \le \infty$, 
	\begin{equation*}
		\| \chi_Q \|_{ L^ {\vec p} }  =|Q|^{\frac{1}{n}  \sum_{i=1}^n   \frac{1}{p_i}} .
	\end{equation*}
\end{lemma}

\begin{example} \label{exa Q0 iff}
	Let $ 0 < \vec p \le \infty $. Let $ 0 <t <\infty $ and $0 < r \le \infty$. Let $Q_0 = [0,1)^n$. Then
	$\chi_{Q_0 }  \in M_{\vec p} ^{t,r} $  if and only if  $   0 < n / ( \sum_{i=1}^n  1/p_{i})    < t <r <\infty $ or $ 0< n / ( \sum_{i=1}^n  1/p_{i})   \le t < r=\infty  $. 
\end{example}
\begin{proof}
	
	Case $r<\infty$. We first calculate the norm $\| \chi_{Q_0 } \|_{M_{\vec p} ^{t,r}  }$. By Lemma \ref{chi Q mixed Lp},	we obtain
	\begin{align*}
		 \sum_{v\ge 0} \sum_{ m \in \mathbb Z^n} \left(|Q_{v, m}|^{ 1/t- \frac{1}{n} \sum_{i=1}^n   \frac{1}{p_i}}   \| \chi_{Q_0 }  \chi_{Q_{v, m} }\|_{L^{\vec p}}    \right) ^r  =  \sum_{v\ge 0} 2^{vn}  2^{-vn  r/t } ,
	\end{align*}
	and
	\begin{align*}
		\sum_{v< 0} \sum_{ m \in \mathbb Z^n} \left(|Q_{v, m}|^{ 1/t- \frac{1}{n} \sum_{i=1}^n   \frac{1}{p_i}}   \| \chi_{Q_0 }  \chi_{Q_{v, m} }\|_{L^{\vec p}}    \right) ^r  =  \sum_{v< 0}  \left(   2^{-vn (  1/t- \frac{1}{n} \sum_{i=1}^n \frac{1}{ p_i}   ) }    \right)^r .
	\end{align*}
	Hence 
	\begin{equation*}
		\| \chi_{Q_0 } \|_{M_{\vec p} ^{t,r}  }  = \left( \sum_{v\ge 0} 2^{vn}  2^{-vn  (r/t) }  +   \sum_{v< 0}  \left(   2^{-vn (  1/t- \frac{1}{n} \sum_{i=1}^n \frac{1}{ p_i}   ) }    \right)^r   \right)^{1/r}.
	\end{equation*}
	The right-hand side is finite if and only if $   0 < n / ( \sum_{i=1}^n  1/p_{i})    < t <r <\infty $.
	
	Case $r  = \infty$. If $v \ge 0$  and $Q_0 \cap Q_{v, m}  \neq \emptyset$, then $  Q_{v, m} \subset Q_0 $ and
	\begin{align*}
		|Q_{v, m}|^{ 1/t- \frac{1}{n} \sum_{i=1}^n   \frac{1}{p_i}}   \| \chi_{Q_0 }  \chi_{Q_{v, m} }\|_{L^{\vec p}} 
		= |Q_{v, m}|^{ 1/t} = 2^{ -vn /t}.
	\end{align*}
	If $v < 0$  and $Q_0 \cap Q_{v, m}  \neq \emptyset$, then $  Q_0  \subset Q_{v, m}$ and
	\begin{align*}
		|Q_{v, m}|^{ 1/t- \frac{1}{n} \sum_{i=1}^n   \frac{1}{p_i}}   \| \chi_{Q_0 }  \chi_{Q_{v, m} }\|_{L^{\vec p}} 
		= 	|Q_{v, m}|^{ 1/t- \frac{1}{n} \sum_{i=1}^n   \frac{1}{p_i}}  = 2^{ -vn  ( 1/t- \frac{1}{n} \sum_{i=1}^n  \frac{1}{ p_i}  )}.
	\end{align*}
	Hence
	\begin{equation*}
		\| \chi_{Q_0 } \|_{M_{\vec p} ^{t,\infty}  }  = \max \left\{  \sup_{v \ge 0, v \in \mathbb Z}  2^{ -vn /t},\sup_{v< 0, v\in \mathbb Z} 2^{ -vn  ( 1/t- \frac{1}{n} \sum_{i=1}^n  \frac{1}{ p_i}  )}  \right \}. 
	\end{equation*}
	The right-hand side is finite if and only if $ 0< n / ( \sum_{i=1}^n  1/p_{i})    \le t < r=\infty  $.
\end{proof}

Now we are ready to obtain the nontriviality of mixed Bourgain-Morrey spaces.

\begin{proposition}
	Let $ 0 < \vec p \le \infty $. Let $ 0 <t <\infty $.
	Then    $ M_{\vec p}^{t,r} \neq \{ 0\}$ if and only if $   0 < n / ( \sum_{i=1}^n  1/p_{i})    < t <r <\infty $ or $ 0< n / ( \sum_{i=1}^n  1/p_{i})    \le t < r=\infty  $.
\end{proposition}

\begin{proof}
	We use the idea from \cite[Theorem 2.10]{HNSH23}.
	From Example \ref{exa Q0 iff}, the ``if'' part is clear. Thus let us prove the ``only if'' part. Since 
$
		\|  |f|^u \|_{   M_{\vec p}^{t,r}  }  =  \|  f \|_{   M_{ u \vec p}^{ ut,ur}  } ^u
$
	for all $f \in L^0$, we can assume $\min \vec p >1$. In this case $ M_{\vec p}^{t,r}$ is a Banach space. Let $f \neq 0 \in  M_{\vec p}^{t,r} $. We may suppose $ f \ge 0 $ by replacing $f$ with $|f|$. By  replacing $f$ with $\min(1, |f|) $, we assume $f \in L^\infty$. Thanks to Corollary \ref{convolution}, $\chi_{[0,1]^n  } * f \in  M_{\vec p}^{t,r} \backslash \{0\}$. Since $\chi_{[0,1]^n  } * f$ is a non-negative non-zero continuous function, there exist $x_0 \in \rn, \epsilon >0 $  and $R>0$  such that $\chi_{[0,1]^n  } * f \ge \epsilon \chi_{B(x_0, R)} $. Thus $ \chi_{B(x_0, R) }  \in M_{\vec p}^{t,r} $. Since $ M_{\vec p}^{t,r}$  is invariant under translation and dilation, $ \chi_{[0,1]^n } \in M_{\vec p}^{t,r} $. Thus from Example \ref{exa Q0 iff}, we obtain $   0 < n / ( \sum_{i=1}^n  1/p_{i})     < t <r <\infty $ or $ 0< n / ( \sum_{i=1}^n  1/p_{i})     \le t < r=\infty  $.
\end{proof}

\subsection{Approximation, density and separability}
Here we investigate the approximation property of $ M_{\vec p}^{t,r} $ when $r<\infty$. 
We first recall the concept of conditional expectation
	from \cite[Definition 2.6.4]{HNVW16}.
	Let $(S,\mathcal A,\mu)$ be a measure space and
	$\mathcal F\subset\mathcal A$ a sub-$\sigma$-algebra.
	A function $g\in L^0(S,\mathcal F, \mu )$ is called a
	conditional expectation  with respect to
	$\mathcal F$  of  the function $f\in L^0(S)$ if
	\[
	\int_F g\,\mathrm d\mu=\int_F f\,\mathrm d\mu
	\]
	for every $F\in\mathcal F$ such that both integrals
	are finite.
	We denote $g$ by $\mathbb E(f | \mathcal F)$.
Then we 
 recall the definition of filtrations, adaptedness and martingales; for example, see \cite[Definition 3.1.1]{HNVW16}.

\begin{definition}
	Let $(S, \mathcal A, \mu )$ be a measure space and $(I, \le) $ an ordered set.
	
	(i) A family of sub-$\sigma$-algebras $\mathscr F$ of $\mathcal  A$ is called a filtration in $(S, \mathcal A, \mu )$ 
	if $\mathscr F_m \subset \mathscr F_n$ whenever $m, n \in I$ and $m \le n$. The filtration is called $\sigma$-finite if $\mu$ is $\sigma$-finite on each $ \mathscr F_n$.
	
	(ii) A family of functions   $ (f_n )_{n \in I} $  in $L^0(S)$ (the measurable functions on $S$) is adapted to the filtration $(\mathscr F_n )_{n \in I} $ if $f_n \in L^0(S,\mathscr F_n  )$   for all $n \in I$.
	
	(iii) A family of functions $ (f_n )_{n \in I} $  in $L^0(S)$  is called a martingale with
	respect to a $\sigma$-finite filtration $(\mathscr F_n )_{n \in I} $ if it is adapted to $(\mathscr F_n )_{n \in I} $, and
	for all indices $m\le n$ the function $f_n$ is $\sigma$-integrable over $\mathscr F_m$ and satisfies
	\begin{equation*}
		\mathbb E(f_n | \mathscr F_m) = f_m.
	\end{equation*}
\end{definition}
The following example comes from \cite[Example 2.6.13, Example 3.1.3]{HNVW16}.
\begin{example}
	Let $(S, \mathcal A, \mu ) = (\rn, \mathcal B (\rn), \d x )$.
	Let $ \mathcal D_j := \{ 2^{-j} ( [0,1)^n +k ) : k\in \mathbb Z^n \}  , j\in \mathbb Z$ be the standard dyadic cubes of side-length $2^{-j} $ as in Section \ref{preliminaries}. Let $\mathscr F_j := \sigma (  \mathcal D_j  )$ its 	generated $\sigma$-algebra. 
	Then $( \mathscr F_j)_{j \in \mathbb Z}$ is a $\sigma$-finite filtration. Every function $f\in L^1_{\operatorname{loc}}$ is $\sigma$-integrable over every $\mathscr F_j$
	and therefore generates a martingale
	$(\mathbb E (f| \mathscr F_j)) _{j\in \mathbb Z}$. In this case, a conditional expectation of $f$ with respect to $\mathscr F_j $
	exists and is given by
	\begin{equation*}
		\sum_{Q \in  \D_j }  \frac{1}{|Q|} \int_Q f (y) \d y   \chi_Q (x), \quad x \in \rn.
	\end{equation*}
\end{example}
For each $k\in\mathbb Z$ and
$f\in L_{\operatorname{loc}}^1$, define the operator
$\mathbb E_k$ by
\begin{equation} \label{def E_k}
	\mathbb E_k   (f) (x) = \sum_{Q \in \D_k}  \frac{1}{|Q|} \int_Q f (y) \d y   \chi_Q (x), \quad x \in \rn.
\end{equation}

By virtue of the Lebesgue differentiation theorem, for almost every $x\in \rn $, 	$\mathbb E_k   (f)(x)$ converges to $f(x)$ as $k \to \infty $. 
We define 
\begin{equation*}
	\mathbb M f (x)= \sup _{ k \in\mathbb Z} \mathbb E_k   ( |f| ) (x) .
\end{equation*}
The following result is the Doob's inequality in mixed Lebesgue spaces.
\begin{lemma}
	[Theorem 2, \cite{SW21}]  \label{doob}
	If $1<\vec p <\infty$ or $\vec p = (\infty, \ldots, \infty ,p_{k+1}, \ldots, p_n)$ with $ 1< p_{k+1}, \ldots, p_n <\infty $ for some $k \in \{1, \ldots, n \} $. Then the maximal operator $\mathbb M$ is bounded on $L^{\vec p}$, that is, for all $f\in L^{\vec p}$
	\begin{equation*}
		\| \mathbb M f \|_{ L^{\vec p}} \lesssim 	\| f \|_{ L^{\vec p}} .
	\end{equation*}
\end{lemma}

For a real-valued measurable function $f$, we set $f_+ := \max (f,0)$ and $f_{-} := \max (-f, 0)$. Then $ \mathbb E_k f = \mathbb E_k f_+  - \mathbb E_k f_-$.

\begin{theorem} \label{E_k^q f to f}
	Let $1 < \vec p <\infty $. 
	Let $1 \le  n / ( \sum_{i=1}^n  1/p_{i})     < t <r <\infty $.
	Then for a real-valued functions  $f \in M_{\vec p}^{t,r} $, the sequence $  \{\mathbb E_k (f_+) -  \mathbb E_k (f_-)  \} $ converges to $f$ in $ M_{\vec p}^{t,r} $.
\end{theorem}
Note that our proof below shows $\mathbb E_k (f_+) - \mathbb E_k (f_-)  $ in $ M_{\vec p}^{t,r} $.
\begin{proof}
	We use the idea from \cite[Theorem 2.20]{HNSH23}.
	Let $f\in M_{\vec p}^{t,r} $ and $\epsilon>0 $ be fixed. 
	Since 
	\begin{equation*}
		\| f - ( \mathbb E_k  (f_+) -  \mathbb E_k (f_-) ) \|_{M_{\vec p}^{t,r}  }  \lesssim \| f_+ -  \mathbb E_k (f_+) \|_{M_{\vec p}^{t,r}  }  +\| f_{-} -   \mathbb E_k (f_-)  \|_{M_{\vec p}^{t,r}  } , 
	\end{equation*}
	we need to show that $ \mathbb E_k (f_+)$ converges to $f_+$ and  $ \mathbb E_k (f_-)$ converges to $f_-$ in  $ M_{\vec p}^{t,r} $. Thus, we may assume that $f \in   M_{\vec p}^{t,r} $ is non-negative and we will prove that $ \{ \mathbb E_k (f_+) \}_{k =1}^\infty $  converges back to $f$ in $M_{\vec p}^{t,r} $.
	
	Since $f \in M_{\vec p}^{t,r}$, there exists $J \in \mathbb N$ large enough such that
	\begin{align} \label{J epsilon}	
		&	\left(\sum_{j\in \mathbb Z}  \sum_{m \in \mathbb Z^n}  ( \chi_{\mathbb Z \backslash [-J,J]} (j)  + \chi_{\mathbb Z^n \backslash B_J  }  (m) )  |Q_{j,m}|^{r/t-\frac{r}{n} \sum_{i=1}^n \frac{1}{ p_i}  } 
		\| \chi_{Q_{j,m}  } f \|_{ L^{\vec p}} ^ r
		\right)^{1/r}  <\epsilon.
	\end{align}
	Fix $k \in \mathbb N \cap (J,\infty)$. Set
	
	\begin{align*}
		I & := \left\| \left\{       |Q_{j,m}|^{1/t-\frac{1}{n} \sum_{i=1}^n \frac{1}{ p_i} } 
		\| ( f -\mathbb E_k (f) ) \chi_{Q_{j,m} } \|_{L^{\vec p}  }
		\right\}_{j \in [-J,J], m\in \mathbb Z^n}  \right\|_{\ell^r} , \\
		II & := \left\| \left\{       |Q_{j,m}|^{1/t-\frac{1}{n} \sum_{i=1}^n \frac{1}{ p_i} } 	\| f \chi_{Q_{j,m} } \|_{L^{\vec p}  } \right\}_{j \in \mathbb Z \backslash [-J,J], m\in \mathbb Z^n}  \right\|_{\ell^r} ,\\
		III & := \left\| \left\{       |Q_{j,m}|^{1/t-\frac{1}{n} \sum_{i=1}^n \frac{1}{ p_i} } \| \mathbb E_k (f)  \chi_{Q_{j,m} } \|_{L^{\vec p}  } \right\}_{j \in ( (-\infty, -J) \cup (J,k]  )\cap \mathbb Z, m\in \mathbb Z^n}  \right\|_{\ell^r}, \\
		IV & := \left\| \left\{       |Q_{j,m}|^{1/t-\frac{1}{n} \sum_{i=1}^n \frac{1}{ p_i} } \| \mathbb E_k (f)  \chi_{Q_{j,m} } \|_{L^{\vec p}  }   \right\}_{j \in  \mathbb Z \backslash (-\infty,k], m\in \mathbb Z^n}  \right\|_{\ell^r}.
	\end{align*}
	Then we decompose
	\begin{equation*}
		\| f-   \mathbb E_k (f) \|_{M_{\vec p}^{t,r}  }  \le I +II+ III+ IV.
	\end{equation*}
	For $I$, using the  Minkowski inequality and (\ref{J epsilon}), we see
	\begin{align*}
		I  \le \left\| \left\{       |Q_{j,m}|^{1/t-\frac{1}{n} \sum_{i=1}^n \frac{1}{ p_i}  } 	\| ( f -\mathbb E_k (f) ) \chi_{Q_{j,m} } \|_{L^{\vec p}  }   \right\}_{j \in [-J,J], m\in \mathbb Z^n \cap B_J}  \right\|_{\ell^r} +\epsilon.
	\end{align*}
Fix $Q_{j,m}$ and set $g=f\chi_{Q_{j,m}}\in L^{\vec p}$.
	For $k\ge j$, we have
	\[
	(\mathbb E_k f)\chi_{ Q_{j,m} }=\mathbb E_k(f\chi_{ Q_{j,m} })=\mathbb E_k g.
	\]
	By the Lebesgue differentiation theorem,
	$\mathbb E_k g\to g$ almost everywhere as $k\to\infty$.
	Moreover, since $f\ge0$, Lemma \ref{doob} yields
	\[
	|g-\mathbb E_k g|\le g+ \mathbb M g\in L^{\vec p}.
	\]
	By dominated convergence in mixed Lebesgue spaces,
	\[
 \lim_{k \to \infty }	\|(f-\mathbb E_k f)\chi_{ Q_{j,m} }\|_{L^{\vec p}}
	=  \lim_{k \to \infty } \|g-\mathbb E_k g\|_{L^{\vec p}} = 0.
	\]
	Since only finitely many cubes occur in the first term
	of the estimate for $I$, we obtain

	\begin{equation*}
		I  < 2\epsilon
	\end{equation*} 
	as long as $k$ large enough.
	
	For $II$,  simply use (\ref{J epsilon}) to conclude $II< \epsilon$.
	
	Next, we estimate $III$. By Lemma \ref{doob},	
	\begin{align*}
		& \| \mathbb E_k (f)  \chi_{Q_{j,m} } \|_{L^{\vec p}  } = \| \mathbb E_k (f \chi_{Q_{j,m} } )  \chi_{Q_{j,m} } \|_{L^{\vec p}  }  \le  \| \mathbb M (f \chi_{Q_{j,m} } )  \chi_{Q_{j,m} } \|_{L^{\vec p}  } \lesssim  \| f \chi_{Q_{j,m} }  \|_{L^{\vec p}  } .
	\end{align*}
	Taking the $\ell^r$-norm over $j \in ( (-\infty, -J) \cup (J,k]  )\cap \mathbb Z, m\in \mathbb Z^n $, we obtain
	\begin{align*}
		III \le \left(   \sum_{ j \in ( (-\infty, -J) \cup (J,k]  )\cap \mathbb Z}   \sum_{ m\in \mathbb Z^n }  |Q_{j,m}|^{r/t-\frac{r}{n} \sum_{i=1}^n \frac{1}{ p_i} } \| f  \chi_{Q_{j,m} } \|_{L^{\vec p}  } ^r \right)^{1/r} <\epsilon.
	\end{align*}
	Finally, we estimate $IV$. Let $j \in \mathbb Z \backslash (-\infty,k]$. Since $Q_{j,m} \subset Q$ for $Q\in \D_k$  as long as $Q_{j,m} \cap Q \neq\emptyset$, using H\"older's inequality with $\vec p >  1$, we have
	\begin{align*}
		| IV |^r & \le  \sum_{j=k+1}^\infty \sum_{m \in \mathbb Z^n}  |Q_{j,m}|^{r/t-\frac{r}{n} \sum_{i=1}^n \frac{1}{ p_i} }  \| \mathbb E_k (f)  \chi_{Q_{j,m} } \|_{L^{\vec p}  } ^r  \\
		& \le \sum_{j=k+1}^\infty   \sum_{ Q \in \D _k} \sum_{m \in \mathbb Z^n,   Q_{j,m} \subset Q}  |Q_{j,m}|^{r/t   } |Q|^{ - \frac{r}{n} \sum_{i=1}^n \frac{1}{ p_i}   } \| f\chi_Q \|_{ L^{\vec p} } ^r .
	\end{align*}
	Note that $   |Q_{j,m}| = 2^{ -(j-k)n } |Q|$. Since $t<r<\infty$, 
	we obtain
	\begin{align*}
		| IV |^r & \le   \sum_{j=k+1}^\infty \sum_{ Q \in \D _k}  \sum_{m \in \mathbb Z^n,   Q_{j,m} \subset Q}  2^{ -(j-k)n r/t }  |Q|^{r/t - \frac{r}{n} \sum_{i=1}^n \frac{1}{ p_i}  } \| f\chi_Q \|_{ L^{\vec p} } ^r  \\
		& \le \sum_{j=k+1}^\infty  2^{ -(j-k)n ( r/t -1  ) }\sum_{  Q\in \D_k}|Q|^{r/t - \frac{r}{n} \sum_{i=1}^n \frac{1}{ p_i}  } \| f\chi_Q \|_{ L^{\vec p} } ^r \\
		& \lesssim \sum_{  Q\in \D_k}|Q|^{r/t - \frac{r}{n} \sum_{i=1}^n \frac{1}{ p_i}  } \| f\chi_Q \|_{ L^{\vec p} } ^r.
	\end{align*}
	Since $J \le k$, from (\ref{J epsilon}), we conclude $IV \lesssim \epsilon$.
	
	Combining the estimates $I -IV$, we get 
	\begin{equation*}
		\limsup_{k\to \infty} \| f- \mathbb E_k (f) \|_{M_{\vec p}^{t,r}}   \lesssim \epsilon.
	\end{equation*} 
	Since $\epsilon>0$ is arbitrary, we  obtain the desired result.
\end{proof}

As a further corollary, the set $L_c^\infty$ of all compactly supported  bounded functions  is dense in $M_{\vec p}^{t,r} $.
\begin{corollary} \label{dense L_c mixed BM}
	Let $1< \vec p <\infty $. 
	Let $ 1 <   n / ( \sum_{i=1}^n  1/p_{i})      < t <r <\infty $. Then $L_c^\infty $  is dense in $M_{\vec p}^{t,r} $.
\end{corollary}
\begin{proof}
	
	We need to approximate $f \in M_{\vec p}^{t,r}$ with functions in $L_c^\infty$.
	By the decomposition
		\[
		f=(\operatorname{Re}f)_+-(\operatorname{Re}f)_-
		+i(\operatorname{Im}f)_+-i(\operatorname{Im}f)_-,
		\]
		it suffices to consider nonnegative functions.
	
 Fix $\epsilon>0$. 
	By Theorem \ref{E_k^q f to f}, there exists $k$ large enough such that 
	\begin{equation} \label{f -Ek f < eps}
		\| f - \mathbb E_k (f) \|_{M_{\vec p}^{t,r} } < \epsilon.
	\end{equation}
	Note that $\mathbb E_k (f)\chi_{B_R}  \in L_c^\infty $  for all $R>0$.
	Fix a large number $k\in \mathbb N$ such that (\ref{f -Ek f < eps}) holds. By the disjointness of the family $\{ Q_{k,m} \}_{m\in \mathbb Z^n}$, 
	\begin{align*}
			\| \mathbb E_k (f) - \mathbb E_k  (f)\chi_{B_R} \|_{ M_{\vec p}^{t,r}  }  
		& = \left( \sum_{j\in \mathbb Z, m\in \mathbb Z^n} |Q_{j,m}|^{r/t-\frac{r}{n} \sum_{i=1}^n \frac{1}{ p_i}  }   \|\mathbb E_k (f) - \mathbb E_k  (f)\chi_{B_R} \|_{L^{\vec p} } ^r   \right)^{1/r}  \\
		& \le  \left( \sum_{j\in \mathbb Z}  \sum_{ m\in \mathbb Z^n, Q_{j,m} \cap B_R^c \neq \emptyset} |Q_{j,m}|^{r/t-\frac{r}{n} \sum_{i=1}^n \frac{1}{ p_i}  }  \|\mathbb E_k (f) - \mathbb E_k  (f)\chi_{B_R} \|_{L^{\vec p} } ^r    \right)^{1/r} .
	\end{align*}
	Note that $\mathbb E_k (f) \in M_{\vec p}^{t,r}$ according to Theorem \ref{E_k^q f to f}. By the Lebesgue convergence theorem, we obtain
	\begin{equation*}
		\| \mathbb E_k (f) - \mathbb E_k  (f)\chi_{B_R} \|_{ M_{\vec p}^{t,r}  }    <\epsilon
	\end{equation*}
	as long as $R$ large enough. 
	Then by the Minkowski inequality, we have
	\begin{equation*}
		\| f - E_k (f)\chi_{B_R} \|_{ M_{\vec p}^{t,r} } \lesssim \epsilon.
	\end{equation*}
	Since $\epsilon>0$ is arbitrary, the proof is complete.
\end{proof}
\begin{remark}
	A consequence of Corollary \ref{dense L_c mixed BM} and Lemma \ref{embed ell r} is that $ M_{\vec p}^{t,r} \subset  \widetilde{ M_{\vec p}^{t,\infty} }  $, where $\widetilde{ M_{\vec p}^{t,\infty} }$ stands for the closure of $L_c^\infty $ in $M_{\vec p}^{t,\infty} $.
\end{remark}

\begin{lemma}[Lusin's property, Theorem 1.14-4 (c), \cite{C13}] \label{Lusin}
	Let $A$ be a subset of $\rn$. Let $f:\rn \to \mathbb R$ be a measurable function with $|A| <\infty$, where $A = \{ x \in \rn: f(x) \neq 0\}$. Then given any $\epsilon >0$, there exists a function $f_\epsilon \in C(\rn)$,  the continuous functions on $\rn$, with compact support  and such that 
	$
	\sup_{x\in \rn } |f_\epsilon (x) | \le \sup_{x\in \rn } |f (x) |
	$
	and $ | \{ x\in \rn :  f(x) \neq f _\epsilon (x) \}| \le \epsilon$.
\end{lemma}

The requirement in [15, Theorem 1.14-4(c)]
	that $\operatorname{supp}f_\epsilon\subset A$ is
	omitted here, since it is not valid for a general
	measurable set $A$.

\begin{corollary} \label{compact continuous dense mixed BM}
	Let $1< \vec p <\infty $. 
	Let $ 1 <   n / ( \sum_{i=1}^n  1/p_{i})     < t <r <\infty $. Then $C_c := C_c (\rn) $, the  set of all compactly supported  continuous functions,  is dense in $M_{\vec p}^{t,r} $.
\end{corollary}
\begin{proof}
	By considering real and imaginary parts,
		it suffices to consider the real-valued functions.	
	Fix $\epsilon>0$ and choose $m_0 \in \mathbb N$ such that $ 1/m_0 \le \epsilon $.
	By Corollary \ref{dense L_c mixed BM}, for $f \in M_{\vec p}^{t,r}$, there exists $g \in L_c^\infty$ such that 
	\begin{equation*}
		\| f-g \|_{M_{\vec p}^{t,r} } \le 1/m_0 \le \epsilon.
	\end{equation*}
Choose $R>0$ such that
	$\operatorname{supp}g\subset B(0,R)$, and choose
	$\eta\in C_c(\mathbb R^n)$ such that $0\le\eta\le1$,
	$\eta=1$ on $B(0,R)$, and
	$\operatorname{supp}\eta\subset B(0,2R)$.
	By Lemma \ref{Lusin}, there exist $h_m\in C_c(\mathbb R^n)$
	such that
	\[
	\|h_m\|_{L^\infty}\le\|g\|_{L^\infty},
	\qquad
	|\{x:g(x)\ne h_m(x)\}|\le\frac1m.
	\]
	Set $g_m=\eta h_m$. Since $\eta g=g$, we have
	\[
	|\{x:g(x)\ne g_m(x)\}|\le\frac1m.
	\]
	Thus, after passing to a subsequence,
	$g_m\to g$ almost everywhere. Moreover,
	\[
	|g-g_m|\le
	2\|g\|_{L^\infty}\chi_{B(0,2R)}.
	\]
	Since $\chi_{B(0,2R)}\in M_{\vec p}^{t,r}$,
	the dominated convergence theorem yields

	\begin{equation*}
		\lim_{ m\to \infty }	\| g -g_m \|_{M_{\vec p}^{t,r} } =0.
	\end{equation*}
	For the above $\epsilon$, choose $m' $ large enough such that $ 	\| g -g_{m'} \|_{M_{\vec p}^{t,r} } \le \epsilon $.
	By the Minkowski inequality, we obtain
	\begin{equation*}
		\| f-g_{m'}  \|_{M_{\vec p}^{t,r} } \le 	\| f-g \|_{M_{\vec p}^{t,r} } +	\| g-g_{m'}  \|_{M_{\vec p}^{t,r} } \le 2\epsilon.
	\end{equation*}
	Thus we finish the proof.
\end{proof}

We say that a metric space $E$ is separable if there exists a subset $ D \subset E$
that is countable and dense.

\begin{theorem} \label{separable mixed BM}
	Let $1< \vec p <\infty $. 
	Let $ 1 <   n / ( \sum_{i=1}^n  1/p_{i})     < t <r <\infty $. Then  $M_{\vec p}^{t,r} $  is separable.
\end{theorem}
\begin{proof}
	By decomposing functions into their real and
		imaginary parts, it suffices to prove separability
		for the real-valued subspace.
	Let $A := \{ I = \prod_{i=1}^n (a_i, b_i): a_i ,b_i \in \mathbb Q \}$. Then $A$ is countable. Let 
	$\mathcal A  $ be the space spanned by functions $\{  \chi_{Q} : Q \in A  \}$ on $\mathbb Q$. That is, $f\in \mathcal A $ has the form $f = \sum_{j \in J} \alpha_j \chi_{Q_j} $ where the set $J$ of indices is finite, and $\alpha _j \in \mathbb Q$ and $Q_j \in A$ for all $j \in J$.
	
	For any $f \in M_{\vec p}^{t,r} $, by Corollary \ref{compact continuous dense mixed BM}, there exists $g\in C_c$ such that $\| f-g\|_{M_{\vec p}^{t,r} } \le \epsilon$.   Choose $I \in A$ such that supp $g \subset I $. Since $g \in C_c$, for any $\epsilon >0$, choose $h \in \mathcal A$  such that $\| h-g\|_{L^\infty} <\epsilon  \| \chi_{I} \|_{M_{\vec p}^{t,r}  } ^{-1}$ and supp $h \subset I $. Then
	$
		\| g -h \|_{M_{\vec p}^{t,r} }  \le \epsilon.
	$
	By the Minkowski inequality, we obtain $ \| f -h\|_{M_{\vec p}^{t,r} }  \le 2\epsilon$.
\end{proof}

To  show the duality of mixed Bourgain-Morrey spaces, we need a smaller dense space $\operatorname{Sim} (\rn) $.

\begin{lemma} \label{dense simple}
	Let $1< \vec p <\infty $. 
	Let $ 1 <   n / ( \sum_{i=1}^n  1/p_{i})     < t <r <\infty $. Then  $\operatorname{Sim} (\rn) $ is dense in $M_{\vec p}^{t,r}$.
\end{lemma}
\begin{proof}
	Without	loss of generality, we may assume that $f \in M_{\vec p}^{t,r}$ is nonnegative. For any $f \in M_{\vec p}^{t,r}$, there exists $g \in L_c^\infty$  such that \begin{equation*}
		\| f-g\|_{ \in M_{\vec p}^{t,r} } \le \epsilon .
	\end{equation*} 
	Using Theorem \ref{E_k^q f to f}, there exists $K \in \mathbb N$ such that 
	\begin{equation*}
		\| g - \mathbb E_K  (g) \|_{M_{\vec p}^{t,r} } \le \epsilon .
	\end{equation*}
	Since $g \in L_c^\infty$, we deduce $\mathbb E_K  (g)  \in \operatorname{Sim} (\rn)$. Then 
	\begin{equation*}
		\| f -   \mathbb E_K  (g) \|_{ M_{\vec p}^{t,r}}  \le 2 \epsilon.
	\end{equation*}
	Thus $\operatorname{Sim} (\rn) $ is dense in $M_{\vec p}^{t,r}$.
\end{proof}

\section{Preduals of mixed Bourgain-Morrey spaces} \label{sec predual}
In this section, we mainly prove  preduals  of mixed Bourgain-Morrey spaces. 
\begin{definition}
	Let $ 1 \le \vec p <\infty$. Let   $n / ( \sum_{i=1}^n  1/p_{i})  \le t \le \infty $. A measurable function $b$ is said to be a
	$(\vec{p}^{\,\prime}, t')$-block if there exists a cube $Q$ that supports $b$ such that 
	\begin{equation*}
		\| b \|_{L^{\vec{p}^{\,\prime} } }  \le |Q| ^{  1/t - \frac{1}{n}  ( \sum_{i=1}^n  1/p_{i})   }  .
	\end{equation*}
	If we need to indicate $Q$, we  say that $b$ is a $(\vec{p}^{\,\prime},t')$-block supported on $Q$.
\end{definition}

\begin{definition}\label{def mix block space}
	Let $ 1 \le \vec p <\infty$. Let   $n / ( \sum_{i=1}^n  1/p_{i})  \le t \le r \le  \infty $.
	The function space $\mathcal{H}_{\vec p \, ^\prime}^{t',r'} $
	is the set of  all measurable functions $f$ such that $f$ is realized
	as the sum
	\begin{equation}\label{eq: mix block f}
		f = \sum_{(j,k)\in\mathbb{Z}^{n+1}}\lambda_{j,k}b_{j,k}
	\end{equation}
	with some $\lambda=\{\lambda_{j,k}\}_{(j,k)\in\mathbb{Z}^{n+1}}\in\ell^{r'}(\mathbb{Z}^{n+1})$
	and $b_{j,k}$ is a $(\vec{p}^{\,\prime},t')$-block supported on  $Q_{j,k}$ where (\ref{eq: mix block f}) converges almost everywhere on $\rn$. The norm of $\mathcal{H}_{\vec p \, ^\prime}^{t',r'} $
	is defined by
	\[
	\|f\|_{\mathcal{H}_{\vec p \, ^\prime}^{t',r'} } :=\inf_{\lambda}\|\lambda\|_{\ell^{r'}},
	\]
	where the infimum is taken over all admissible sequence $\lambda$
	such that (\ref{eq: mix block f}) holds.
\end{definition}

\begin{remark} \label{block norm 1}
	Let  $g_{j_0, k_0}$ be a  $(\vec{p}^{\,\prime},t')$-block supported on some $Q_{j_0, k_0}$. 
	Then $ \|g_{j_0, k_0} \|_{\mathcal{H}_{\vec p \, ^\prime}^{t',r'} }  \le 1$. Indeed, let 
	\begin{equation*}
		\lambda_{j,k} = \begin{cases}
			1, &  \operatorname{if} j = j_0, k = k_0; \\
			0 , & \operatorname{else}.
		\end{cases}
	\end{equation*} 
	Hence  $ \|g_{j_0, k_0} \|_{\mathcal{H}_{\vec p \, ^\prime}^{t',r'} }  \le  \left( \sum_{(j,k)\in   \mathbb Z ^{n+1}  } |	\lambda_{j,k} |^{r'}  \right)^{1/r'} =1.  $
\end{remark}

\begin{proposition}\label{mix block sum converge}
	Let $ 1 \le \vec p <\infty$. Let   $n / ( \sum_{i=1}^n  1/p_{i})  < t < r \le  \infty $.
	Assume that $ \{\lambda_{j,k}\}_{(j,k)\in\mathbb{Z}^{n+1} }  \in \ell^{r'}$ and for each $(j,k)\in\mathbb{Z}^{n+1}$, $b_{j,k}$  is a  $(\vec p \, ^\prime,t')$-block supported on $Q_{j,k}$. Then
	
	{\rm (i)} the summation (\ref{eq: mix block f}) converges  almost everywhere   on $\rn$;
	
	{\rm (ii)} the summation (\ref{eq: mix block f}) converges   in $L^1_{\operatorname{loc}  }$. 
\end{proposition}
\begin{proof}
	The proof is similar to \cite[Propositon 3.4]{ZSTYY23} and we omit it here.
\end{proof}

\begin{remark}
	Based on Proposition \ref{mix block sum converge}, we find that, in Definition \ref{def mix block space}, instead of the requirement that   the summation (\ref{eq: mix block f}) converges  almost everywhere on $\rn$, if we require that (\ref{eq: mix block f}) converges in $L^1_{\operatorname{loc}  }$, we then obtain the same block space.
\end{remark}

\begin{lemma} \label{block = Lp mixed}
	Let $ 1 <  \vec p <\infty$. Let   $n / ( \sum_{i=1}^n  1/p_{i})  = t < r =  \infty $.
	Then  $\mathcal{H}_{\vec p \, ^\prime}^{t',1}  = L^{\vec p \, ^\prime} $ with equivalent 
	 norms.
\end{lemma}

\begin{proof}
	We use the idea from	\cite[Proposition 339]{SDH20}.
	If $f \in \mathcal{H}_{\vec p \, ^\prime}^{t',1} $. Then there exist a sequence $ \{ \lambda_{j,k}\} _{  (j,k)  \in \mathbb Z^{1+n} } \in  \ell^1  $ and a sequence $\{ b_{j,k} \}_{(j,k)  \in \mathbb Z^{1+n}  }$ of  $(\vec p \, ^\prime,t') $-blocks such that  $f = \sum_{ (j,k) \in \mathbb Z^{1+n}  }   \lambda_{j,k}  b_{j,k} $  and
	\begin{equation*}
		\sum_{  (j,k) \in \mathbb Z^{1+n} }    | \lambda_{j,k}|    \le  (1+\epsilon)	\| f\|_{\mathcal{H}_{\vec p \, ^\prime}^{t',1}  }.
	\end{equation*}
	Then by Minkowski's inequality, we have
	\begin{align*}
		\left\| f  \right\|_{ L^{\vec p \, ^\prime}  }  
		& \le \sum_{ (j,k) \in \mathbb Z^{1+n}  }   | \lambda_{j,k}  | \|  b_{j,k} \|_{ L^{\vec p \, ^\prime}  } 
		\le 	\sum_{  (j,k) \in \mathbb Z^{1+n} }    | \lambda_{j,k}|    \le  (1+\epsilon)	\| f\|_{\mathcal{H}_{\vec p \, ^\prime}^{t',1}  } .
	\end{align*}	
	For the another direction, let $f \in  L^{\vec p \, ^\prime} $. Fix $\epsilon>0$. Then there is a decomposition 
	\begin{equation*}
		f = \chi_{Q_1} f + \sum_{j=2}^\infty  \chi_{Q_j \backslash Q_{j-1}} f, 
	\end{equation*}
	where $ \{ Q_j\}_{j=1}^\infty  $ is an increasing sequence of cubes centered at the origin satisfying 
	\begin{equation*}
		\| \chi_{\rn \backslash  Q_k } f \|_{  L^{\vec p \, ^\prime}  } \le \epsilon 2^{-k}
	\end{equation*}
	for all $k \in \mathbb N$. Then using this decomposition, we have
	\begin{equation*}
		\| f\|_{\mathcal{H}_{\vec p \, ^\prime}^{t',1}  }  \lesssim  \| f\|_{  L^{\vec p \, ^\prime}    }  + \sum_{k=1}^\infty \epsilon 2^{-k}.
	\end{equation*}
	Thus $f \in \mathcal{H}_{\vec p \, ^\prime}^{t',1} $.  Thus the proof is complete.
\end{proof}

The following lemma indicates how to generate blocks.
\begin{lemma} \label{generate mix block}
	Let $ 1 < \vec p <\infty$. Let   $1 < n / ( \sum_{i=1}^n  1/p_{i})   < t <r<\infty $ or $ 1< n / ( \sum_{i=1}^n  1/p_{i})  \le t< r =\infty  $.
	Let $f$ be an $L^{\vec  p \; ^\prime } $ function supported	on a cube $Q \in \D$. Then $\| f\|_ { \H_{\vec p \, ^\prime}^{t',r'} }  \le \| f\| _{ L^{\vec p \, ^\prime} }   |Q| ^{ \frac{1}{n}  ( \sum_{i=1}^n  1/p_{i})  -  1/t  }   $.
\end{lemma}
\begin{proof}
	Assume that $f$ is not zero almost everywhere, let 
	\begin{equation*}
		b := \frac{   |Q| ^{  1/t - \frac{1}{n}  ( \sum_{i=1}^n  1/p_{i})   }    }{ \|f\|_{L^{\vec p \, ^\prime}    }  } f.
	\end{equation*}
	Then $b$ is supported on $Q \in \D$, and 
	\begin{equation*}
		\| b \|_{L^{\vec  p'}  }  =     |Q| ^{  1/t - \frac{1}{n}  ( \sum_{i=1}^n  1/p_{i})   }     .  
	\end{equation*}
	Hence $b$ is a $ (\vec p \, ^\prime,t')$-block and $\|b\|_{  \H_{\vec p \, ^\prime}^{t',r'} }  \le  1$. 
	Equating this inequality, we obtain the desired result.
\end{proof}

\begin{theorem} \label{dense block}
	Let $ 1 < \vec p <\infty$. Let   $1 < n / ( \sum_{i=1}^n  1/p_{i})   < t <r<\infty $ or $ 1< n / ( \sum_{i=1}^n  1/p_{i})  \le t< r =\infty  $. Then $L^{\vec p \, ^\prime}  _c $ is dense in $\H_{\vec p \, ^\prime}^{t',r'}  $.
\end{theorem}
\begin{proof}
	We use the idea from \cite[Theorem 345]{SDH20}.
	Since $ f \in \H_{\vec p \, ^\prime}^{t',r'} $, there exist a sequence $\{ \lambda_{j,k} \}_{ (j,k) \in \mathbb Z ^{1+n} }  \in \ell ^{r'} $  and a sequence  $(\vec p \, ^\prime,t')$-block $\{  b_{j,k} \}  _{ (j,k) \in \mathbb Z ^{1+n} } $ 
	such that $f 	=\sum_{(j,k)\in\mathbb{Z}^{n+1}}\lambda_{j,k}b_{j,k} $  and $\| \{ \lambda_{j,k} \} \|_{\ell ^{r'}}  \le \| f\|_{\H_{\vec p \, ^\prime}^{t',r'} }    +\epsilon$.	
	Define 
	\begin{equation} \label{f_N dense}
		f_N = \sum_{|(j,k)|_\infty  \le N}\lambda_{j,k}b_{j,k},
	\end{equation}
	where $|(j,k)|_\infty  = \max  \{ |j|, |k_1|, \ldots, |k_n| \}  $.
	Since $r' \in [1,\infty)$, we get
	
		\[
		\|f-f_N\|_{\H_{\vec p'}^{t',r'}}
		\le
		\left(
		\sum_{|(j,k)|_\infty>N}
		|\lambda_{j,k}|^{r'}
		\right)^{1/r'}
		\to 0		
		\]

	as $N\to \infty.$ It is not hard to show $f_N \in L^{\vec p \, ^\prime } $,  (for example, see \cite[Theorem 4.3]{BGX25}).	 
	Hence
	$L^{\vec p \, ^\prime}  _c $ is dense in $\H_{\vec p \, ^\prime}^{t',r'} $. 
\end{proof}

From the above theorem, when we investigate $  \H_{\vec p \, ^\prime}^{t',r'}$, the space $L^{\vec p \, ^\prime}  _c $ is a useful space. When considering the action of  linear operators defined  and continuous on    $ \H_{\vec p \, ^\prime}^{t',r'}$ and $L^{\vec p \, ^\prime}  _c $, it will be helpful to have a finite decomposition in $L^{\vec p \, ^\prime}  _c  $.
The following result says that each function $f \in L^{\vec p \, ^\prime}  _c $ has a finite admissible expression. That is, the sum is finite.
\begin{theorem}\label{finite decom}
	Let $ 1 < \vec p <\infty$. Let   $1 < n / ( \sum_{i=1}^n  1/p_{i})   < t <r<\infty $ or $ 1< n / ( \sum_{i=1}^n  1/p_{i})  \le t< r =\infty  $. Then each $ f \in L^{\vec p \, ^\prime}  _c $ admits the finite decomposition $ f  = \sum_{v=1}^M \lambda_v b_v $ where $ \lambda_1, \lambda_2, \ldots, \lambda_M \ge 0 $ and each $ b_v $ is a $(\vec p \, ^\prime,t')$-block. Furthermore, 
	\begin{equation*}
		\| f\|_{\H_{\vec p \, ^\prime}^{t',r'}} \approx \inf_\lambda \left( \sum_{v=1}^M \lambda_v ^{r'} \right)^{1/r'},
	\end{equation*}
	where $ \lambda = \{ \lambda_v\}_{v=1}^M $ runs over all finite admissible expressions as in Definition 4.2:
	\begin{equation*}
		f= \sum_{v=1}^M \lambda_v b_v,
	\end{equation*}
	$ \lambda_1, \lambda_2, \ldots, \lambda_M \ge 0 $ and each $ b_v $ is a $(
	\vec p \, ^\prime,t')$-block for each $v = 1,2,\ldots ,M$.
\end{theorem}
\begin{proof}
	The proof is similar to \cite[Theorem 4.4]{BGX25} and we omit it here.	
\end{proof}

\begin{remark}
	Let $ 1 < \vec p <\infty$. Let   $1 < n / ( \sum_{i=1}^n  1/p_{i})   < t <r<\infty $ or $ 1< n / ( \sum_{i=1}^n  1/p_{i})  \le t< r =\infty  $. Then  $\H_{\vec p \, ^\prime}^{t',r'}  $ is separable. The proof is similar with Theorem \ref{separable mixed BM}.
\end{remark}

In \cite[Theorem 2.7]{N192}, Nogayama showed that the predual space of  $M_{\vec p}^{t,\infty} $ is $\mathcal{H}_{\vec p \, ^\prime}^{t',1} $.
Before proving the following result,   we give notation for the mixed Lebesgue norm $\| \cdot \|_{L^{\vec p} }$. The mapping
\begin{equation*}
	(x_2, \ldots , x_n) \in \mathbb R^{n-1} \mapsto \| f\|_{ (p_1) }  (x_2, \ldots , x_n) := \left(\int_{\mathbb R}  |f (x_1, x_2, \ldots, x_n) |^{p_1} \d x_1 \right) ^{1/p_1}
\end{equation*}
is a measurable function  on  $\mathbb R^{n-1}$. Moreover, let
\begin{equation*}
	\| f \|_{ \vec q } =\| f \|_{ (p_1, \ldots, p_j) } := \left\|  [ \| f\|_{ (p_1, \ldots, p_{j-1})  } ]  \right\|_{  (  p_j ) },
\end{equation*}
where $\| f \|_{ (p_1, \ldots, p_{j-1}) }$ denotes $|f|$ if $j=1$  and $\vec q = (p_1, \ldots, p_j), j\le n$.
Then $ \| f \|_{ \vec q  }$  is a measurable function of $ (x_{j+1}, \ldots, x_n )$ for $j <n$.

\begin{theorem} \label{predual mix BM}
	Let $ 1 < \vec p <\infty$. Let   $1 < n / ( \sum_{i=1}^n  1/p_{i})   < t <r<\infty $ or $ 1< n / ( \sum_{i=1}^n  1/p_{i})  \le t< r =\infty  $.
	Then the dual  space of
	$\mathcal{H}_{\vec p \, ^\prime}^{t',r'} $, denoted by $ \Big( \mathcal{H}_{\vec p \, ^\prime}^{t',r'} \Big)^* $, is
	$M_{\vec p}^{t,r} $, that is,  
	\begin{equation*}
		\Big(  \mathcal{H}_{\vec p \, ^\prime}^{t',r'} \Big)^*   = M_{\vec p}^{t,r} 
	\end{equation*}
	in the following sense:
	
	{\rm (i)} if $f \in M_{\vec p}^{t,r} $, then the linear functional 
	\begin{equation} \label{Jf}
		L_f : f \to  L_f (g) :=   \int_\rn   g (x) f (x)  \d x
	\end{equation}
	is bounded on $ \mathcal{H}_{\vec p \, ^\prime}^{t',r'} $.
	
	{\rm (ii)} conversely, any continuous linear functional on $\mathcal{H}_{\vec p \, ^\prime}^{t',r'} $ arises as in  (\ref{Jf}) with a unique $f \in  M_{\vec p}^{t,r}  $.
	
	Moreover, $\|f\|_{   M_{\vec p}^{t,r}  }  = \| L_f \|_{ ( \mathcal{H}_{\vec p \, ^\prime}^{t',r'}  )^*}$ and 
	\begin{equation} \label{H = max M le 1}
		\|g\|_{\mathcal{H}_{\vec p \, ^\prime}^{t',r'}   } =\max \left\{\left|\int_\rn  g (x) f (x)  \d x    \right|: f\in   M_{\vec p}^{t,r} , \|f\|_{   M_{\vec p}^{t,r} } \le 1 \right\}.
	\end{equation}
\end{theorem}
\begin{proof}
	(i) Let $g \in \mathcal{H}_{\vec p \, ^\prime}^{t',r'} $ and $\epsilon>0$. Then $g = \sum_{(j,k)\in\mathbb{Z}^{n+1}}\lambda_{j,k}b_{j,k}$ with $\lambda=\{\lambda_{j,k}\}_{(j,k)\in\mathbb{Z}^{n+1}}\in\ell^{r'}(\mathbb{Z}^{n+1})$
	and $b_{j,k}$ is a  $(\vec p \, ^\prime,t')$-block supported on $Q_{j,k}$ such that 
	\begin{equation*}
		\left( \sum_{(j,k)\in\mathbb{Z}^{n+1}} |\lambda_{j,k}|^{r'}\right)^{1/r'}  \le (1+\epsilon) \| g\|_{ \mathcal{H}_{\vec p \, ^\prime}^{t',r'} }.
	\end{equation*}
	By  H\"older's inequality,
	\begin{align} \label{f g le M H}
		\nonumber
		\left| \int_\rn g (x) f (x)  \d x\right| &  = 	\left|	\int_\rn   \sum_{(j,k)\in\mathbb{Z}^{n+1}}\lambda_{j,k}b_{j,k} (x)   f (x)  \d x\right|  \le \sum_{(j,k)\in\mathbb{Z}^{n+1}} |\lambda_{j,k}| |Q_{j,k}  | ^{  1/t - \frac{1}{n}  ( \sum_{i=1}^n  1/p_{i})   }  \| f   \chi_{Q_{j,k}}     \|_{L^{\vec p } }    \\
		& \le \left( \sum_{(j,k)\in\mathbb{Z}^{n+1}} |\lambda_{j,k}|^{r'}\right)^{1/r'} \| f\|_{M_{\vec p}^{t,r}  }  \le (1+\epsilon) \| g\|_{  \mathcal{H}_{\vec p \, ^\prime}^{t',r'}  }  \| f\|_{M_{\vec p}^{t,r}   }.
	\end{align}
	Letting $\epsilon\to 0^+$,  we prove (i).
	
	(ii) Let $L$ be a continuous linear functional on $\mathcal{H}_{\vec p \, ^\prime}^{t',r'}$.	
	By Lemma \ref{generate mix block}, since we can regard the element of $ L^{\vec p \, ^\prime} (Q_{j,k})$ as a $ (\vec p \, ^\prime,t')$-block modulo multiplicative constant, the functional $g \mapsto L(g) $  is well defined and bounded on $ L^{\vec p \, ^\prime} (Q_{j,k})$. Thus, by the $L^{\vec p} -L^{\vec{p}^{\,\prime}} $  duality (Lemma \ref{dual mixed Lp}), there exists $f_{j,k} \in  L^{\vec p} (Q_{j,k})$ such that
	\begin{equation} \label{L g jk}
		L(g) = \int_{Q_{j,k}} f_{j,k} (x) g(x) \d x
	\end{equation}
	for all $g \in L^{\vec p \, ^\prime} (Q_{j,k})$. By the uniqueness of this theorem, we can find $  L^{\vec p}_{\operatorname{loc}} $-function $f$ such that 
	\begin{equation*}
		f \chi_{Q_{j,k}}  = f_{j,k} \quad \operatorname{a.e.}
	\end{equation*}
	for any $Q_{j,k} \in \D $. We shall prove $f \in M_{\vec p}^{t,r}$. 
	
	Fix  a finite set $K \subset \mathbb Z^{n+1}$. For each $(j,k) \in K$, 
	let 
	\begin{equation*}
		g_{j,k} := 
		\frac{|Q_{j,k} |^{  1/t - \frac{1}{n}  ( \sum_{i=1}^n  1/p_{i})  } }{ \| f\chi_{ Q_{j,k} } \|_{L^{\vec p}  } ^{ p_n  -1} } (\overline{ \sgn f}) |f|^{p_1 -1} \chi_{ Q_{j,k} }  \| f \chi_{ Q_{j,k} } \|_{( p_1 ) }^{ p_2 -p_1  } \cdots \| f \chi_{ Q_{j,k} } \|_{( p_1, \ldots , p_{n-1}  ) }^{  p_n - p_{n-1}   }   
	\end{equation*}
	if $ \| f\chi_{ Q_{j,k} } \|_{L^{\vec p}  } >0 $  and let $g_{j,k} = 0$ if   $ \| f\chi_{ Q_{j,k} } \|_{L^{\vec p}  } =0 $.
	Then if $ \| f\chi_{ Q_{j,k} } \|_{L^{\vec p}  } >0 $,
	\begin{equation} \label{int f g jk}
		\int_{Q_{j,k}} f (x) g_{j,k}(x) \d x = |Q_{j,k} |^{  1/t - \frac{1}{n}  ( \sum_{i=1}^n  1/p_{i})  } 	\| f\chi_{ Q_{j,k} }\|_{L^{\vec p}  } ,
	\end{equation}
	and 
	\begin{equation*}
		\| g_{j,k} \|_{L^{\vec{p}^{\,\prime} }  }  =  |Q_{j,k} |^{  1/t - \frac{1}{n}  ( \sum_{i=1}^n  1/p_{i})  } .
	\end{equation*}
	Note that $   g_{j,k} $  is a $(\vec p \, ^\prime , t') $-block supported on $Q_{j,k}$.
	Take an arbitrary nonnegative
	sequence $\{  \lambda_{j,k}\} \in \ell^{r'} (\mathbb Z^{1+n})$
	supported on $K$ and set
	\begin{equation}\label{g K decom}
		g_{K}  = \sum_{(j,k) \in K}  \lambda_{j,k}   g_{j,k}  \in \mathcal{H}_{\vec p \, ^\prime}^{t',r'}  .
	\end{equation}
	Then  from  (\ref{L g jk}), (\ref{int f g jk}), the fact that $K$ is a finite set, and the linearity of $L$, we get 
	\begin{align*}
		\sum_{(j,k) \in K}  \lambda_{j,k}  |Q_{j,k}| ^{1/t - \frac{1}{n}  ( \sum_{i=1}^n  1/p_{i}) } \left\|   f \chi_{Q_{j,k}}    \right\|_{L^{\vec p} }  = \sum_{(j,k) \in K}  \lambda_{j,k}  \int_{Q_{j,k}} f (x) g_{j,k}(x) \d x   =   \int_\rn   g_{K} (x) f(x ) \d x = L (g_{K}  ).
	\end{align*}
	By the decomposition (\ref{g K decom}) and $L$ is a continuous linear functional on $ \mathcal{H}_{ \vec p \, ^\prime}^{t',r'} $, we obtain
	\begin{equation*}
		L (g_{K}  )  \le  \| L \|_{ (\mathcal{H}_{\vec p \, ^\prime}^{t',r'})^*  }  \| g_{K} \|_{  \mathcal{H}_{\vec p \, ^\prime}^{t',r'}  } \le \| L \|_{ (\mathcal{H}_{\vec p \, ^\prime}^{t',r'})^*  } \left( \sum_{(j,k) \in K}   \lambda_{j,k} ^ {r'} \right)^{1/r'}.
	\end{equation*}
	Since $r >1$ and $K  \subset \mathbb Z^{1+n}$ and $ \{\lambda_{j,k} \}_{(j,k) \in K } $ are arbitrary, we conclude 
	\begin{align*}
		 \| f \|_{M_{\vec p}^{t,r}  }   = \sup \left\{  \sum_{ (j,k) \in \mathbb Z^{1+n}  }  \lambda_{j,k} |Q_{j,k}| ^{1/t - \frac{1}{n}  ( \sum_{i=1}^n  1/p_{i}) } \left\|   f \chi_{Q_{j,k}}    \right\|_{L^{\vec p} } :  \| \{\lambda_{j,k} \} \|_{\ell^{r' } (\mathbb Z^{1+n})}  = 1     \right\} 
	 \le   \| L \|_{ (\mathcal{H}_{\vec p \, ^\prime}^{t',r'})^*  }  < \infty .
	\end{align*}
	Thus $f \in M_{\vec p}^{t,r}$ and $\| f \|_{M_{\vec p}^{t,r}  } \le  \| L \|_{ (\mathcal{H}_{\vec p \, ^\prime}^{t',r'})^*  } $. Together (i), we obtain $\| f \|_{M_{\vec p}^{t,r}  } =  \| L \|_{ (\mathcal{H}_{\vec p \, ^\prime}^{t',r'})^*  } $ 
	
	Hence we conclude that $L$  is realized as $L = L_f$ for $f \in M_{\vec p}^{t,r}$ at least on $g \in L_c^{\vec p \, ^\prime}   $. Since $ L_c^{\vec p \, ^\prime}  $  is dense in $\mathcal{H}_{\vec p \, ^\prime}^{t',r'}$ by Theorem \ref{dense block}, we can obtain the desired result.

	Next we show that $f$ is unique. Suppose that there exists another $\tilde f \in M_{\vec p}^{t,r}$ such that $L$ arises as in (\ref{Jf}) with $f $ replaced by $\tilde f$.
	Then for any $ Q \in \D$, since $\chi_Q \overline{\sgn ( f  - \tilde f   ) }  \in  \mathcal{H}_{\vec p \, ^\prime}^{t',r'} $ , we have
	\begin{equation*}
		\int_Q	 \overline{\sgn ( f(x) - \tilde f (x) ) } ( f(x) - \tilde f (x) )  \d x = 	\int_Q	 | f(x) - \tilde f (x) |  \d x =  0.
	\end{equation*}
	This, combined with the arbitrariness of $ Q \in \D$ further implies that $f=  \tilde f $ almost everywhere.
	
	Using  \cite[Theorem 87, Existence of the norm attainer]{SDH20}, (\ref{H = max M le 1}) is included in what we have proven.
	Thus we complete the proof.
\end{proof}

\section{Properties of  block spaces}\label{property block}

In this section, we discuss the completeness,  Fatou property, lattice and associate space of 
the function space $\mathcal{H}_{\vec p \, ^\prime}^{t',r'} $.

We first show the completeness.

\begin{theorem} \label{Banach}
	Let $ 1 < \vec p <\infty$. Let   $1 < n / ( \sum_{i=1}^n  1/p_{i})   < t <r<\infty $ or $ 1< n / ( \sum_{i=1}^n  1/p_{i})  \le t< r =\infty  $. Then $\mathcal{H}_{\vec p \, ^\prime}^{t',r'} $ is a Banach space.
\end{theorem}

\begin{proof}
	The proof is similar to \cite[Theorem 4.1]{BGX25} and we omit it here.
\end{proof}

Next we consider the Fatou property of block spaces.

\begin{theorem} \label{Fatou general}
	Let $ 1 < \vec p <\infty$. Let   $1 < n / ( \sum_{i=1}^n  1/p_{i})   < t <r<\infty $ or $ 1< n / ( \sum_{i=1}^n  1/p_{i})  \le t< r =\infty  $.
	If a bounded sequence
	$\{f_\ell\}_{\ell\in\mathbb N}
	\subset\mathcal H_{\vec p'}^{t',r'}
	\cap L_{\operatorname{loc}}^1$
	converges locally to $f$ in the weak topology of
	$L^1$, then
 $f \in \H_{\vec p \, ^\prime}^{t',r'}$ and 
	\begin{equation*}
		\|f\|_{ \H_{\vec p \, ^\prime}^{t',r'} }  \le \liminf _{\ell \to \infty} \|f_\ell\|_{ \H_{\vec p \, ^\prime}^{t',r'} } .
	\end{equation*}
\end{theorem}

\begin{proof}
	If $t=n/(\sum_{i=1}^n1/p_i)$, then $r=\infty$.
		By Lemma 4.6,
		$\mathcal H_{\vec p \, ^\prime }^{t',1}=L^{\vec p \, ^\prime }$
		with equivalent norms.
	
	In the remaining cases, we follow the proof of
		\cite[Theorem 4.6]{BGX25}, using mixed Lebesgue norms. 
	
The convergence assumption is only used to obtain
	$\int_Q f_\ell\to\int_Q f$ for each fixed dyadic cube $Q$,
	which follows from local weak convergence in $L^1$.
\end{proof}

Then we use the Fatou property of $\H_{\vec p \, ^\prime}^{t',r'} $ to show the following result.
\begin{corollary}
	Let $ 1 < \vec p <\infty$. Let   $1 < n / ( \sum_{i=1}^n  1/p_{i})   < t <r<\infty $ or $ 1< n / ( \sum_{i=1}^n  1/p_{i})  \le t< r =\infty  $.  Suppose that  $ fg \in L^1 $ for all $f  \in  M_{\vec p}^{t,r} $. Then $g \in \H_{\vec p \, ^\prime}^{t',r'}  $.
\end{corollary}
\begin{proof}
	By the closed graph theorem, $ \|  fg \|_{L^1}   \le L \|f\|_{M_{\vec p}^{t,r}}  $  for some $L$ independent of $f  \in M_{\vec p}^{t,r}$. Set   $g_N := \chi_{B(0,N)}  \chi_{ [0,N] } (|g|) g $ for each $N \in \mathbb N$. Then using Lemma \ref{generate mix block} and the fact that $B(0,N)$ can be covered by at most $2^n$ dyadic cubes, we have $g_N \in \H_{\vec p \, ^\prime}^{t',r'} $. Using Theorem \ref{predual mix BM}, we obtain $ \| g_N \| _{ \H_{\vec p \, ^\prime}^{t',r'}  }  \le L. $ Since $\chi_B\in M_{\vec p}^{t,r}$
		for every ball $B$, the assumption implies
		$g\in L_{\operatorname{loc}}^1$.
		As $|g_N|\le|g|$ and $g_N\to g$ almost everywhere,
		dominated convergence gives
		$g_N\to g$ in $L_{\operatorname{loc}}^1$.
		By Theorem \ref{Fatou general},
		$g\in\mathcal H_{\vec p'}^{t',r'}$.
	
\end{proof}

The following result is the  lattice property of block spaces $\H_{\vec p \, ^\prime}^{t',r'} $.
\begin{lemma} \label{lem lattice block}
	Let $ 1 < \vec p <\infty$. Let   $1 < n / ( \sum_{i=1}^n  1/p_{i})   < t <r<\infty $ or $ 1< n / ( \sum_{i=1}^n  1/p_{i})  \le t< r =\infty  $.
	Then a measurable function $ f$  belongs to $\H_{\vec p \, ^\prime}^{t',r'}  $ if and only if there exists a measurable non-negative function  $ g \in \H_{\vec p \, ^\prime}^{t',r'} $ such that  $|f(x)| \le g (x)$ a.e.
\end{lemma}

\begin{proof}
	The proof is similar to \cite[Lemma 4.10]{BGX25} and we omit it here.
\end{proof}

Finally we research the associate spaces and dual spaces.
Recall ball Banach function norm and ball Banach function space.
\begin{definition} \label{ball Bf norm}
	A mapping $\rho: \mathbb M ^+ \to [0,\infty]$  is called a ball Banach  function norm (over $\rn$) if, for all $f, g, f_k (k \in \mathbb N)$  in $\mathbb M ^+ $, for all constants $a \ge 0$ and for all balls $B$ in $\rn$, the	following properties hold:
	
	{\rm (P1)} $\rho (f) = 0 $ if and only if $f =0$ a.e., $\rho (a f) = a \rho (f)$, $\rho(f+g) \le \rho (f) +\rho (g)$,
	
	{\rm (P2)} $\rho (g) \le \rho (f)$  if $ g \le f $ a.e.,
	
	{\rm (P3)} the Fatou property; $\rho (f_j) \uparrow \rho (f)$ holds whenever $0 \le f_j \uparrow f$ a.e.,
	
	{\rm (P4)} For all balls $B$, $\rho (\chi_B) <\infty$,
	
	{\rm (P5)} $\| \chi_B f \|_{L^1} \lesssim_B \rho (f) $ with the implicit constant depending on $B$ and $\rho$ but independent of $f.$
	
	The space generated by such $\rho$ is called a ball Banach function	space.
	
	A mapping $\tau: \mathbb M ^+ \to [0,\infty]$  is called a  Banach  function norm (over $\rn$) 	if it satisfies (P1), (P2), (P3),   
	and the following conditions
		for every measurable set $E\subset\mathbb R^n$
		with $|E|<\infty$:

{\rm (P4)$^\prime$}  $\tau (\chi_E) <\infty$,

{\rm (P5)$^\prime$} $\| \chi_E f \|_{L^1} \lesssim_E \tau (f) $ with the implicit constant depending on $E$ and $\tau$ but independent of $f$.

	The space generated by a  Banach  function norm is called a  Banach function	space.
\end{definition}

\begin{definition}
	If $\rho$  is  a ball Banach function norm, its associate norm $\rho '$ is defined in  $ \mathbb M ^+$  by
	\begin{equation*}
		\rho ' (g) := \sup\left\{ \| fg\|_{L^1} :f \in \mathbb M ^+, \rho (f) \le 1 \right\}, g \in  \mathbb M ^+ .
	\end{equation*}
\end{definition}

\begin{definition}
	Let $\rho$ be a ball Banach function norm, and let $ \mathcal X = \mathcal X (\rho) $ be the ball Banach function space determined by $\rho$ as in Definition \ref{ball Bf norm}. Let $\rho '$ be the associate norm of $\rho$. The Banach function space $\mathcal X(\rho ') = \mathcal X ' (\rho)$ determined by $\rho '$ is called the associate space or the K\"othe dual of $\mathcal X$.
\end{definition}

\begin{lemma}[Theorem 352, \cite{SDH20}] \label{E= E''}
	Every ball Banach function space $E(\rn)$ coincides with its
	second associate space $E'' (\rn)$. In other words, a function $f$ belongs to $E(\rn)$
	if and only if it belongs to $E ''(\rn)$, and in that case $\|f\|_{E} = \|f\|_{E''}$ or all
	$f \in E (\rn) = E'' (\rn)$.
\end{lemma}

\begin{theorem} \label{associate space}
	Let $ 1 < \vec p <\infty$. Let   $1 < n / ( \sum_{i=1}^n  1/p_{i})   < t <r<\infty $ or $ 1< n / ( \sum_{i=1}^n  1/p_{i})  \le t< r =\infty  $.
	Then the associate space $( M_{\vec p}^{t,r} )  ' $ as a ball Banach function space coincides with the block space $\mathcal{H}_{\vec p \, ^\prime}^{t',r'}  $.
\end{theorem}

\begin{proof}
	We use the idea from \cite[Theorem 355]{SDH20}. 
	As we have seen in Theorem \ref{predual mix BM}, $\mathcal{H}_{\vec p \, ^\prime}^{t',r'} \subset  ( M_{\vec p}^{t,r} )  '  $. Hence, we will verify the converse. Suppose that $f\in L^0$ satisfies 
	\begin{equation} \label{fg le 1}
		\sup \left\{ \| fg\|_{L^1} : \|g\|_{ M_{\vec p}^{t,r} } \le  1 \right\}  \le 1.
	\end{equation}
	Then we first see that $|f(x)|<\infty$, a.e. $x\in \rn$. Splitting $f$ into its real
	and imaginary parts and each of these into its positive and negative parts,
	we may assume without loss of generality that $f \in \mathbb  M^+.$ For $k \in \mathbb N$, set $f_k := \min(f,k) \chi_{B_k }$. By Theorem \ref{predual mix BM}, Lemma \ref{lem lattice block} and (\ref{fg le 1}), we have $f_k \in \mathcal{H}_{\vec p \, ^\prime}^{t',r'}$ and $ \| f_k \|_{\mathcal{H}_{\vec p \, ^\prime}^{t',r'}} \le 1$. 	
	Since $f_k \uparrow f$ a.e., by Theorem \ref{Fatou general}, we obtain $f \in \mathcal{H}_{\vec p \, ^\prime}^{t',r'}$ and $ \| f\|_{\mathcal{H}_{\vec p \, ^\prime}^{t',r'}} \le 1$. 
	
	By Theorem \ref{Fatou general} and Lemma \ref{E= E''}, we have $ \mathcal{H}_{\vec p \, ^\prime}^{t',r'} = (\mathcal{H}_{\vec p \, ^\prime}^{t',r'})''$. This proves the theorem.
\end{proof}

Theorem \ref{associate space} and Lemma \ref{E= E''} yield $( \mathcal{H}_{\vec p \, ^\prime}^{t',r'} ) ' = ( M_{\vec p}^{t,r} )  '' = M_{\vec p}^{t,r}  $.

\begin{definition}[Absolutely continuous norm] 
	Let $ 0 < \vec p \le \infty $ and  $ 0 <t <\infty $.	
	
	A function $f \in M_{\vec p}^{t,r} $ has absolutely continuous norm in $ M_{\vec p}^{t,r}$ if  $\| f\chi_{E_k}\| _{M_{\vec p}^{t,r}} \to 0  $ 
	for every sequence $\{ E_k \}_{k=1}^\infty $ satisfying $E_k \to \emptyset $ a.e., in the sense that
	$\lim_{k \to \infty} \chi_{E_k} =0$ for almost all $x\in \rn$.
\end{definition}

\begin{lemma} \label{lem Absolutely continuous norm}
	Let $1< \vec p < \infty $ and 
	$ 1 <  n / ( \sum_{i=1}^n  1/p_{i})    < t <r <\infty $.  For $f \in  M_{\vec p}^{t,r}$, $f$ has the
	absolutely continuous norm.
\end{lemma}
\begin{proof}
	Let $f \in M_{\vec p}^{t,r}$. By Lemma \ref{dense simple}, we can find a sequence of simple functions $\{f_k\}_{k\in \mathbb N} \subset \operatorname{Sim} (\rn)$ such that 
	\begin{equation*}
		\lim_{k\to \infty} \| f- f_k\|_{M_{\vec p}^{t,r}}  =0.
	\end{equation*}
	Let $\{F_k \}_{k\in \mathbb N}$ be an arbitrary sequence for which $F_k \to \emptyset$ a.e. Then
	\begin{equation*}
		\|f \chi_{F_j} \|_{ M_{\vec p}^{t,r}} \le \| f_k -f \|_{M_{\vec p}^{t,r}}  + \| f_k  \chi_{F_j} \|_{ M_{\vec p}^{t,r}} .
	\end{equation*}
	Letting $j \to \infty$, we obtain
	\begin{equation*}
		\limsup_{j\to \infty} 	\|f \chi_{F_j} \|_{ M_{\vec p}^{t,r}} \le \| f_k -f \|_{M_{\vec p}^{t,r}} .
	\end{equation*}
	Since $k$ is arbitrary, we have $	\lim_{j\to \infty} 	\|f \chi_{F_j} \|_{ M_{\vec p}^{t,r}} =0$.
\end{proof}

The following result comes from \cite[Theorem 1.2.18]{ML19} if we replace `ball Banach function space' by 	`Banach function space'. 
\begin{lemma} \label{acn dual}
	Let $X $ be a  ball Banach function space. Then the following 	conditions are equivalent:
	
	{\rm (i)}  $X^\ast = X '$;
	
	{\rm (ii)} $X$ has absolutely continuous norm.
\end{lemma}

By Theorem \ref{associate space}  and Lemmas \ref{lem Absolutely continuous norm}, \ref{acn dual}, we obtain the following result.

\begin{theorem} \label{dual mixed BM}
	Let $1< \vec p < \infty $ and 
	$ 1 <  n / ( \sum_{i=1}^n  1/p_{i})    < t <r <\infty $.  
	Then the dual of $  M_{\vec p}^{t,r} $ is $ \H_{\vec p \, ^\prime}^{t',r'}  $ in the following sense:
	
	If $g \in \H_{\vec p \, ^\prime}^{t',r'} $, then $L_g : f \mapsto \int_\rn f(x) g(x) \d x $ is an element of $ (  M_{\vec p}^{t,r} ) ^\ast  $. Moreover, for any $L \in  (  M_{\vec p}^{t,r} ) ^\ast$, there exists $g \in\H_{\vec p \, ^\prime}^{t',r'}$ such that $L = L_g$.
\end{theorem}

\begin{corollary}\label{reflexive BMX}
	Let $1< \vec p < \infty $ and 
	$ 1 <  n / ( \sum_{i=1}^n  1/p_{i})     < t <r <\infty $.     Then  $   M_{\vec p}^{t,r} $ is reflexive.
\end{corollary}
\begin{remark}
	Let $\vec p = p$, then Theorem \ref{dual mixed BM} becomes \cite[Theorem 5.5]{HNSH23} and Corollary \ref{reflexive BMX}  becomes \cite[Remark 5.7]{HNSH23}.
\end{remark}

\section{Operators on mixed Bourgain-Morrey spaces and their preduals} \label{sec operator}
In this section, we discuss some operators on mixed Bourgain-Morrey spaces and their preduals.
\subsection{The Hardy-Littlewood maximal operator on mixed Bourgain-Morrey spaces} 
In this subsection, we consider the boundedness property of the Hardy-Littlewood maximal operator and iterated maximal operator on mixed Bourgain-Morrey spaces.

For the boundedness of Hardy-Littlewood maximal operator,
we define a equivalent norm of $M_{\vec p}^{t,r}  $ below. To do so, we use a dyadic grid $\mathcal D_{k,\vec a}$, $k \in \mathbb Z$, $\vec a \in \{ 0,1,2\}^n$. More precisely, let
\begin{equation*}
	\mathcal D_{k,\vec a}^0  :=  \{ 2^{-k} [ m +a/3, m+a/3 +1) : m\in \mathbb Z   \}
\end{equation*}
for $k \in \mathbb Z$  and $a = 0,1,2$. Consider
\begin{equation*}
	\mathcal D_{k,\vec a} := \{ Q_1 \times Q_2 \times \cdots \times Q_n :Q_j \in	\mathcal D_{k,\vec a_j}^0 ,  j =1,2,\ldots , n \}
\end{equation*}
for $k \in \mathbb Z$  and  $\vec a = (a_1, a_2, \ldots, a_n) \in  \{ 0,1,2\}^n$. Hereafter,  a  dyadic grid is the family $\mathcal D _{\vec a} := \cup_{k\in \mathbb Z} \mathcal D_{k,\vec a} $ for $\vec a \in \{0,1,2\}^n$.

The following result is an important property of the dyadic grids to prove the equivalent norm of the mixed Bourgain-Morrey space.
\begin{lemma}[Lemma 2.8, \cite{HNSH23}] \label{cube be covered}
	For any cube $Q$ there exists $R\in \bigcup_{\vec a \in \{0,1,2\}^n}  \mathcal D _{\vec a}$ such that $Q \subset R$  and $|R| \le 6^n |Q|$.
\end{lemma}

Fix the dyadic grid $  \mathcal D _{\vec a}$ where $ \vec a \in \{0,1,2\}^n  $ and  define
\begin{equation*}
	\| f \|_{  M_{\vec p}^{t,r} (\mathcal D _{\vec a}) } : = \left\|\left\{|Q|^{  \frac{1}{t}  - \frac1n\sum_{i=1}^n\frac1{p_i} }  \| f \chi_Q \|_{L^{\vec p }  }  \right\}_{Q \in \mathcal D _{\vec a} }    \right\|_{\ell^r}
\end{equation*}
for all $f \in L^0 $.

Using Lemma \ref{cube be covered}, we obtain the following result. Since its proof is similar to  \cite[ (2.1)]{HNSH23},  we omit it here.
\begin{lemma} \label{D equivalence}
	Let $ 0 < \vec p \le \infty $. Let    $0 < n / ( \sum_{i=1}^n  1/p_i )   < t <r <\infty $ or $ 0< n / ( \sum_{i=1}^n  1/p_i )  \le t < r=\infty  $.
	Then $ \|\cdot \|_{M_{\vec p}^{t,r}   }$  and $ \|\cdot \|_{M_{\vec p}^{t,r} (\mathcal D _{\vec a})  } $ are equivalent for any dyadic grid $\mathcal D _{\vec a} $ where $ \vec a \in \{0,1,2\}^n  $. That is, for each $f \in L^0 $, $ \|f\|_{M_{\vec p}^{t,r}   }   \approx   \|f\|_{M_{\vec p}^{t,r} (\mathcal D _{\vec a})  }  $.
\end{lemma}

Denote by $\mathcal M_{\operatorname{dyadic}}$ the dyadic maximal operator generated by the dyadic cubes  in $\mathcal D$. We do not have a pointwise estimate to control $\mathcal M$ in terms of $\mathcal M_{\operatorname{dyadic}}$. 
The maximal operator generated by a family $\mathcal D _{\vec a}$ is defined by 
\begin{equation*}
	\mathcal M_{\mathcal D _{\vec a}} f (x) = \sup_{Q \in \mathcal D _{\vec a}} \frac{\chi_Q (x)}{|Q|} \int_Q |f(y) | \d y
\end{equation*}
for $f\in L^0 $ and $ \vec a \in \{0,1,2\}^n  $.
As in \cite{LN19},
\begin{equation} \label{M le Mdya}
	\mathcal M f (x) \lesssim \sum_{ \vec a \in \{0,1,2\}^n } \M_{\mathcal D _{\vec a}} f (x).
\end{equation}

\begin{theorem} \label{HL M}
	Let $ 1< \vec p < \infty $. Let    $1 < n / ( \sum_{i=1}^n  1/p_i )   < t <r <\infty $ or $ 1< n / ( \sum_{i=1}^n  1/p_i )  \le t < r=\infty  $.
	Then $\mathcal M$  is bounded on  $M_{\vec p}^{t,r} $.
\end{theorem}

\begin{proof}
	We only need to prove the case $r<\infty$, since
	$r=\infty$ is in \cite[Theorem 4.5]{N19}.
	Due to (\ref{M le Mdya}) and Lemma \ref{D equivalence},
	it is sufficient to show the
	boundedness of $\mathcal M_{\mathcal D_{\vec a}}$.
	Let $f\in M_{\vec p}^{t,r}$ and
	$Q\in\mathcal D_{\vec a}$. Set
	$f_1=f\chi_{3Q}$ and
	$f_2=f\chi_{\rn\backslash3Q}$.
	By \cite[Theorem 4.5]{N19}, we have
	\begin{equation}\label{M f_1}
		\|\mathcal M_{\mathcal D_{\vec a}}f_1\|_{L^{\vec p}(Q)}
		\lesssim
		\|f_1\|_{L^{\vec p}}
		=
		\|f\|_{L^{\vec p}(3Q)}.
	\end{equation}	
	Let
		\[
		\mathcal N(Q):=
		\{P\in\mathcal D_{\vec a}:
		\ell(P)=\ell(Q),\ P\cap3Q\neq \emptyset \}.
		\]
		Hence
		\[
		\|f\|_{L^{\vec p}(3Q)}
		\le
		\sum_{P\in\mathcal N(Q)}
		\|f\|_{L^{\vec p}(P)}.
		\]
		The cardinality of $\mathcal N(Q)$ is bounded by
		a constant depending only on $n$. Moreover, each
		$P$ belongs to $\mathcal N(Q)$ for at most a bounded
		number of cubes $Q$.
		Since $|P|=|Q|$ for $P\in\mathcal N(Q)$, using
		(\ref{M f_1}), we obtain
		\[
		\left(
		\sum_{Q\in\mathcal D_{\vec a}}
		|Q|^{r(1/t-\frac1n\sum_{i=1}^n\frac1{p_i} )}
		\|\mathcal M_{\mathcal D_{\vec a}}f_1\|_{L^{\vec p}(Q)}^r
		\right)^{1/r}
		\lesssim
		\|f\|_{M_{\vec p}^{t,r}(\mathcal D_{\vec a})}.
		\]

	For $k\ge1$, define
		\[
		\mathcal R_k(Q):=
		\{R\in\mathcal D_{\vec a}:
		\ell(R)=2^k\ell(Q),\ R\cap Q\ne\varnothing\}.
		\]
		If $x\in Q$ and $R\in\mathcal D_{\vec a}$ contains
		$x$ with $\ell(R)\le\ell(Q)$, then $R\subset3Q$,
		and hence
		\[
		\int_R|f_2(y)|\,dy=0.
		\]
		Therefore, by the H\"older inequality
		(Lemma \ref{Holder mixed}), for $x\in Q$ we obtain
		\[
		\mathcal M_{\mathcal D_{\vec a}}f_2(x)
		\le
		\sum_{k=1}^{\infty}
		\sum_{R\in\mathcal R_k(Q)}
		\frac1{|R|}\int_R|f(y)|\,dy
		\le
		\sum_{k=1}^{\infty}
		\sum_{R\in\mathcal R_k(Q)}
		|R|^{-\frac1n\sum_{i=1}^n\frac1{p_i} }\|f\|_{L^{\vec p}(R)}.
		\]
		Since $\|\chi_Q\|_{L^{\vec p}}=|Q|^{ \frac1n\sum_{i=1}^n\frac1{p_i} } $
		and $|R|=2^{kn}|Q|$ for $R\in\mathcal R_k(Q)$,
		we have
		\begin{align*}
			&|Q|^{1/t-\frac1n\sum_{i=1}^n\frac1{p_i} }
			\|\mathcal M_{\mathcal D_{\vec a}}f_2\|_{L^{\vec p}(Q)}
		\le
			\sum_{k=1}^{\infty}2^{-kn/t}
			\sum_{R\in\mathcal R_k(Q)}
			|R|^{1/t-\frac1n\sum_{i=1}^n\frac1{p_i} }\|f\|_{L^{\vec p}(R)}.
		\end{align*}
			
	For fixed $Q$ and $k$, the cardinality of
		$\mathcal R_k(Q)$ is bounded by a constant depending
		only on $n$. For fixed $R$ and $k$, the number of
		cubes $Q$ satisfying $R\in\mathcal R_k(Q)$ is at most
		$C_n2^{kn}$. Consequently,
		\begin{align*}
			\sum_{Q\in\mathcal D_{\vec a}}
			\left(
			\sum_{R\in\mathcal R_k(Q)}
			|R|^{1/t-\frac1n\sum_{i=1}^n\frac1{p_i} }\|f\|_{L^{\vec p}(R)}
			\right)^r
			&\qquad\lesssim
			\sum_{Q\in\mathcal D_{\vec a}}
			\sum_{R\in\mathcal R_k(Q)}
			|R|^{r(1/t-\frac1n\sum_{i=1}^n\frac1{p_i} )}\|f\|_{L^{\vec p}(R)}^r
			\\
			&\qquad\lesssim
			2^{kn}
			\sum_{R\in\mathcal D_{\vec a}}
			|R|^{r(1/t-\frac1n\sum_{i=1}^n\frac1{p_i} )}\|f\|_{L^{\vec p}(R)}^r
			\\
			&\qquad=
			2^{kn}\|f\|_{M_{\vec p}^{t,r}(\mathcal D_{\vec a})}^r.
		\end{align*}
		By Minkowski's inequality, we obtain
		\begin{align*}
			&\left(
			\sum_{Q\in\mathcal D_{\vec a}}
			|Q|^{r(1/t-\frac1n\sum_{i=1}^n\frac1{p_i} )}
			\|\mathcal M_{\mathcal D_{\vec a}}f_2\|_{L^{\vec p}(Q)}^r
			\right)^{1/r}
			\\
			&\quad\le
			\sum_{k=1}^{\infty}2^{-kn/t}
			\left(
			\sum_{Q\in\mathcal D_{\vec a}}
			\left(
			\sum_{R\in\mathcal R_k(Q)}
			|R|^{1/t-\frac1n\sum_{i=1}^n\frac1{p_i} }\|f\|_{L^{\vec p}(R)}
			\right)^r
			\right)^{1/r}
			\\
			&\quad\lesssim
			\sum_{k=1}^{\infty}
			2^{-kn/t}2^{kn/r}
			\|f\|_{M_{\vec p}^{t,r}(\mathcal D_{\vec a})}
			\\
			&\quad\lesssim
			\|f\|_{M_{\vec p}^{t,r}(\mathcal D_{\vec a})},
		\end{align*}
		where the last inequality follows from $t<r$.
		Combining the estimates of $\mathcal M_{\mathcal D_{\vec a}}f_1$  and $ \mathcal M_{\mathcal D_{\vec a}}f_2 $,
		we obtain
	\[
	\|\mathcal M_{\mathcal D_{\vec a}}f
	\|_{M_{\vec p}^{t,r}(\mathcal D_{\vec a})}
	\lesssim
	\|f\|_{M_{\vec p}^{t,r}(\mathcal D_{\vec a})}
	\]
	as desired.	
\end{proof}

We  define the maximal operator$ \M_{(k)}$ for  $x_k$  as follows:
\begin{equation*}
	\M_{(k)} f (x) := \sup_{x_k \in I} \frac1{|I|} \int _I  | f(x_1, \ldots, y_k ,\ldots, x_n) | \d y_k,
\end{equation*}
where interval $I$ ranges over all intervals containing $x_k$. Furthermore, for all measurable functions
$f$, define the iterated maximal operator  $\M^{(\operatorname{it})}$ by
\begin{equation*}
	\M^{(\operatorname{it})} f(x) := \M_{(n)} \cdots \M_{(1)} (|f|) (x) .
\end{equation*}

\begin{proposition} \label{mixed BM equ norm weight}
	Let $ 1< \vec p < \infty $. Let    $ 1 < n / ( \sum_{i=1}^n  1/p_i )   < t <r <\infty $ or $ 1 < n / ( \sum_{i=1}^n  1/p_i )  \le t < r=\infty  $. Let $\eta \in (0,1)  $  such that
	\begin{equation*}
		0 \le    \left(\sum_{i=1 } ^n \frac{1}{p_i} -\frac{n}{t}  +  \frac{n}{r} \right)  <\eta <1 .
	\end{equation*}
	Then for $f\in L^0(\rn)$, we have
	\begin{equation*}
		\|f\|_{M_{\vec p}^{t,r}  }  \approx  \left(\sum_{Q \in \D} |Q|^{ \frac{r}{t}  - \frac{r}{n} \sum_{i=1 } ^n \frac{1}{p_i}  }  \| f ( \M^{\operatorname{it} } \chi_Q) ^\eta \|_{L^{\vec p}  } ^r  \right)^{1/r} .
	\end{equation*}
\end{proposition}
\begin{proof}
	We use the idea from \cite[Proposition 4.10]{N19}.
	We only need to show the case $r<\infty$  since  case $r=\infty$ is similar. 
	From $ \chi_Q \le  ( \M^{\operatorname{it} } \chi_Q) ^\eta$, we obtain
	\begin{equation*}
		\|f\|_{M_{\vec p}^{t,r}  }  \le \left(\sum_{Q \in \D} |Q|^{ \frac{r}{t}  - \frac{r}{n} \sum_{i=1 } ^n \frac{1}{p_i}  }  \| f ( \M^{\operatorname{it} } \chi_Q) ^\eta \|_{L^{\vec p}  } ^r  \right)^{1/r} .
	\end{equation*}
	For the opposite inequality, fix a cube $Q = I_1 \times \cdots \times I_n$. Given  $(\ell_1, \ldots, \ell_n) \in \mathbb N^n$, we write $\ell = \max (\ell_1, \ldots, \ell_n)$. Then we get
	\begin{align*}
			|Q|^{ \frac{1}{t}  - \frac{1}{n} \sum_{i=1 } ^n \frac{1}{p_i}  }  \| f ( \M^{\operatorname{it} } \chi_Q) ^\eta \|_{L^{\vec p}  }
		& \lesssim |Q|^{ \frac{1}{t}  - \frac{1}{n} \sum_{i=1 } ^n \frac{1}{p_i}  }  \left\| f  \prod_{j=1}^n  \left( \frac{ \ell (I_j)}{\ell(I_j) + |  \cdot_j - c(I_j) | |}\right) ^\eta  \right\|_{L^{\vec p}  } \\
		& \lesssim |Q|^{ \frac{1}{t}  - \frac{1}{n} \sum_{i=1 } ^n \frac{1}{p_i}  } 
		\sum_{\ell_1, \ldots, \ell_n =1}^\infty  \frac{1}{2^{ (\ell_1 + \ldots + \ell_n )\eta  } }
		\left\| f  \chi_{ 2^{\ell_1} I_1 \times \cdots \times  2^{\ell_n} I_n }    \right\|_{L^{\vec p}  } \\
		& \lesssim 
		\sum_{\ell_1, \ldots, \ell_n =1}^\infty  \frac{  2^{\ell n (\frac{1}{n} \sum_{i=1 } ^n \frac{1}{p_i} -\frac{1}{t}  ) }   }{2^{ (\ell_1 + \ldots + \ell_n )\eta  } } 
		|2^\ell Q|^{ \frac{1}{t}  - \frac{1}{n} \sum_{i=1 } ^n \frac{1}{p_i}  } 
		\left\| f  \chi_{ 2^{\ell} Q }    \right\|_{L^{\vec p}  } ,
	\end{align*}
	where $c(I_j)$ denote the center of $I_j$.
	To justify the next estimate, cover each
		$2^\ell Q$ by at most $2^n$ dyadic cubes of the same
		side length. Each covering cube occurs at most
		$C_n2^{\ell n}$ times as $Q$ ranges over $\mathcal D$.
		Thus
		\[
		\left(
		\sum_{Q\in\mathcal D}
		|2^\ell Q|^{r/t-\frac rn\sum_{i=1}^n1/p_i}
		\|f\|_{L^{\vec p}(2^\ell Q)}^r
		\right)^{1/r}
		\lesssim
		2^{\ell n/r}\|f\|_{M_{\vec p}^{t,r}}.
		\]	
	Consequently, adding
	 over all cubes $Q\in \D$, we obtain
	\begin{align*}
		 \left(\sum_{Q \in \D} |Q|^{ \frac{r}{t}  - \frac{r}{n} \sum_{i=1 } ^n \frac{1}{p_i}  }  \| f ( \M^{\operatorname{it} } \chi_Q) ^\eta \|_{L^{\vec p}  } ^r  \right)^{1/r}  
		& \lesssim \sum_{\ell_1, \ldots, \ell_n =1}^\infty  \frac{  2^{\ell n (\frac{1}{n} \sum_{i=1 } ^n \frac{1}{p_i} -\frac{1}{t}  ) }   }{2^{ (\ell_1 + \ldots + \ell_n )\eta  } }  \left(\sum_{Q \in \D}  |2^\ell Q|^{ \frac{r}{t}  - \frac{r}{n} \sum_{i=1 } ^n \frac{1}{p_i}  } 
		\left\| f  \chi_{ 2^{\ell} Q }    \right\|_{L^{\vec p}  } ^r    \right)^{1/r}
		\\
		& \lesssim \sum_{\ell_1, \ldots, \ell_n =1}^\infty  \frac{  2^{\ell n (\frac{1}{n} \sum_{i=1 } ^n \frac{1}{p_i} -\frac{1}{t}  ) }   }{2^{ (\ell_1 + \ldots + \ell_n )\eta  } }  2^{\ell n/r}  \| f\|_{ M_{\vec p}^{t,r} }  \lesssim  \| f\|_{ M_{\vec p}^{t,r} } .
	\end{align*}
The last series converges since
$\eta>\sum_{i=1}^n1/p_i-n/t+n/r$.
Thus the proof is complete.
\end{proof}
\begin{remark}
	We give some explanations about the sum $\sum_{\ell_1, \ldots, \ell_n =1}^\infty  \frac{  2^{\ell n a }   }{2^{ (\ell_1 + \ldots + \ell_n )b  } } $.
	Let $a>0$ and $b>0$.		
		Let $\ell=\max(\ell_1,\ldots,\ell_n)$.
		If $\ell_1=\ell$ and $\ell_2=\cdots=\ell_n=1$, then
		\[
		\frac{2^{\ell na}}{2^{(\ell_1+\cdots+\ell_n)b}}
		=2^{(1-n)b+\ell(na-b)}.
		\]
		Conversely, if $b-na>0$, we have
		\[
		\sum_{\ell_1,\ldots,\ell_n=1}^{\infty}
		\frac{2^{\ell na}}{2^{(\ell_1+\cdots+\ell_n)b}}
		\le
		2^{(1-n)b}
		\sum_{\ell=1}^{\infty}
		\ell^n2^{-\ell(b-na)}
		<\infty.
		\]
\end{remark}

\begin{lemma}[Proposition 4.9, \cite{N19}] \label{Mit mixed weight}
	Let $1<\vec p <\infty$. Let $f \in L^0$ and $ \omega_k \in A_{p_k} (\mathbb R) $  for $k =1 ,\ldots ,n.$ Then
	\begin{equation*}
		\left\| \M^{\operatorname{it}} f \cdot \bigotimes_{k=1}^n \omega_k ^{1/p_k} \right\|_{L^{\vec p}} \lesssim \left\| f \cdot \bigotimes_{k=1}^n \omega_k ^{1/p_k} \right\|_{L^{\vec p}} .
	\end{equation*}
\end{lemma}

\begin{theorem}
	Let $ 1< \vec p < \infty $. Let    $ 1 < n / ( \sum_{i=1}^n  1/p_i )   < t <r <\infty $ or $ 1 < n / ( \sum_{i=1}^n  1/p_i )  \le t < r=\infty  $. 
	Furthermore, let 		
	\begin{equation*}
		\left( \sum_{i=1 } ^n \frac{1}{p_i} -\frac{n}{t}  +  \frac{n}{r} \right) <  \frac{1}{\max \vec p}.
	\end{equation*}
	Set $\eta \in (0,1)  $  such that
	\begin{equation*}
		0 \le  \left ( \sum_{i=1 } ^n \frac{1}{p_i} -\frac{n}{t}  +  \frac{n}{r} \right)  <\eta  < \frac{1}{\max \vec p}<1 .
	\end{equation*}
	Then 
	\begin{equation*}
		\|  \Mit f \|_{ M_{\vec p}^{t,r}   } \lesssim \| f\|_{ M_{\vec p}^{t,r} } .
	\end{equation*}
\end{theorem}
\begin{proof}
	Let $Q = I_1 \times \cdots \times I_n$. Then
	\begin{equation*}
		( \Mit \chi_Q) ^{\eta }  = \left(  \bigotimes_{i=1}^n \M _{(i)} \chi_{I_i} \right)^{\eta }  =  \bigotimes_{i=1}^n ( \M _{(i)} \chi_{I_i} )^{\eta } .
	\end{equation*}
	From \cite[Theorem 288]{SDH20}, $ ( \M _{(i)} \chi_{I_i} )^{\eta  p_i} $  belongs to $A_1(\mathbb R)$ if  $\eta > 0$ such that
	\begin{equation*}
		0 < \eta  p_i < 1 .
	\end{equation*}
	Hence $ ( \M _{(i)} \chi_{I_i} )^{\eta  p_i} \in A_1(\mathbb R) \subset A_{p_i}(\mathbb R) $ for all $p_i$. By Lemma \ref{Mit mixed weight}, we obtain
	\begin{align*}
		\| \Mit f ( \Mit \chi_Q) ^{\eta } \|_{L^{ \vec p } }  =  \left\| ( \Mit f ) \bigotimes_{i=1}^n ( \M _{(i)} \chi_{I_i} )^{\eta }    \right\|_{ L^{ \vec p } } 
		\lesssim  \left\|  f \bigotimes_{i=1}^n ( \M _{(i)} \chi_{I_i} )^{\eta }    \right\|_{ L^{ \vec p } } = \left\|  f   ( \Mit  \chi_Q )^{\eta }    \right\|_{ L^{ \vec p } } .
	\end{align*}
	Then by Proposition \ref{mixed BM equ norm weight},
	we have
	\begin{align*}
		\|  \Mit f \|_{ M_{\vec p}^{t,r}   }  & \approx \left(\sum_{Q \in \D} |Q|^{ \frac{r}{t}  - \frac{r}{n} \sum_{i=1 } ^n \frac{1}{p_i}  }  \| \Mit f ( \M^{\operatorname{it} } \chi_Q) ^\eta \|_{L^{\vec p}  } ^r  \right)^{1/r}  \\
		& 	\lesssim \left(\sum_{Q \in \D} |Q|^{ \frac{r}{t}  - \frac{r}{n} \sum_{i=1 } ^n \frac{1}{p_i}  }  \left\|  f   ( \Mit  \chi_Q )^{\eta }    \right\|_{ L^{ \vec p } }  ^r  \right)^{1/r} 
	\approx \| f\|_{ M_{\vec p}^{t,r}   } .
	\end{align*}
	We finish the proof.
\end{proof}

The vector-valued case is important because it is useful in establishing the theory of related function spaces. 
\begin{definition}
	Let $ 0 < \vec p \le \infty $. Let $   0 < n / ( \sum_{i=1}^n  1/p_{i})     < t <r <\infty $ or $ 0< n / ( \sum_{i=1}^n  1/p_{i})    \le t < r=\infty  $. Let $ 0< u \le \infty$.
	The mixed vector-valued  Bourgain-Morrey norm $\| \cdot \|_{ M_{\vec p} ^{t,r}  (\ell^u) } $  is defined by 
	\begin{align*}
		\| \{f_k \}_{k\in \mathbb N} \|_{  M_{\vec p} ^{t,r}  (\ell^u) } 
		& :=  \left(\sum_{Q \in \D} |Q|^{ \frac{r}{t}  - \frac{r}{n} \sum_{i=1 } ^n \frac{1}{p_i}  }  \left\| \left( \sum_{k=1}^\infty  |f_k|^u \right)^{1/u}  \chi_Q \right\|_{L^{\vec p}  } ^r  \right)^{1/r} ,
	\end{align*}
	where the sequence $\{f_k \}_{k=1}^\infty \subset L^0 $. The mixed vector-valued   Bourgain-Morrey space $  M_{\vec p} ^{t,r}  (\ell^u)  $ is the set of all measurable sequences  $\{f_k \}_{k\in \mathbb N}$ with finite norm $ \|\{f_k \}_{k\in \mathbb N}\|_{  M_{\vec p} ^{t,r}  (\ell^u)  }$.
	
\end{definition}

\begin{theorem} \label{HL seq mix BM}
	Let $ 1< \vec p <  \infty $ and $1<u \le \infty $. Let    $1 < n / ( \sum_{i=1}^n  1/p_i )   < t <r <\infty $ or $ 1< n / ( \sum_{i=1}^n  1/p_i )  \le t < r=\infty  $.
	Then for all $\{f_k \}_{k\in \mathbb N} \in M_{\vec p} ^{t,r}  (\ell^u)  $, we have
	\begin{equation*}
		\|\{\M f_k \}_{k\in \mathbb N}\|_{  M_{\vec p} ^{t,r}  (\ell^u) }  \lesssim 	\|\{f_k \}_{k\in \mathbb N}\|_{  M_{\vec p} ^{t,r}  (\ell^u) } .
	\end{equation*}
\end{theorem}
\begin{proof}
	When $r=\infty$, the result follows from
		\cite[Theorem 1.8]{N19}.
		Thus we assume $r<\infty$.
		For $u=\infty$, the assertion follows from
		Theorem \ref{HL M} and
		\[
		\sup_{k\in\mathbb N}\mathcal M f_k
		\le
		\mathcal M\left(\sup_{k\in\mathbb N}|f_k|\right).
		\]
		Hence it remains to consider $1<u<\infty$.
		By (\ref{M le Mdya}) and Lemma \ref{D equivalence},
		it suffices to estimate
		$\mathcal M_{\mathcal D_{\vec a}}$ for each
		$\vec a\in\{0,1,2\}^n$.
	
	For $Q\in\mathcal D_{\vec a}$, split
		\[
		f_k=f_k\chi_{3Q}+f_k\chi_{(3Q)^c}.
		\]
		By the mixed-norm vector-valued maximal inequality
		\cite[Theorem 1.7]{N19},
		\[
		\left\|
		\left(
		\sum_{k=1}^{\infty}
		|\mathcal M_{\mathcal D_{\vec a}}(f_k\chi_{3Q})|^u
		\right)^{1/u}\chi_Q
		\right\|_{L^{\vec p}}
		\lesssim
		\left\|
		\left(\sum_{k=1}^{\infty}|f_k|^u\right)^{1/u}
		\chi_{3Q}
		\right\|_{L^{\vec p}}.
		\]
		Using the same covering and counting argument
		for $3Q$ as in the proof of Theorem \ref{HL M},
		we obtain
		\begin{align*}
		\left(
			\sum_{Q\in\mathcal D_{\vec a}}
			|Q|^{r/t-\frac rn\sum_{i=1}^n1/p_i}
			\left\|
			\left(
			\sum_{k=1}^{\infty}
			|\mathcal M_{\mathcal D_{\vec a}}(f_k\chi_{3Q})|^u
			\right)^{1/u}\chi_Q
			\right\|_{L^{\vec p}}^r
			\right)^{1/r}
		\lesssim
			\|\{f_k\}_{k\in\mathbb N}
			\|_{M_{\vec p}^{t,r}(\ell^u)}.
		\end{align*}
		
	For $m\ge1$, let
		\[
		\mathcal R_m(Q):=
		\{R\in\mathcal D_{\vec a}:
		\ell(R)=2^m\ell(Q),\ R\cap Q\neq \emptyset\}.
		\]
		For $x\in Q$, we have
		\[
		\mathcal M_{\mathcal D_{\vec a}}
		(f_k\chi_{(3Q)^c})(x)
		\le
		\sum_{m=1}^{\infty}
		\sum_{R\in\mathcal R_m(Q)}
		\frac1{|R|}\int_R|f_k(y)|\,\mathrm d y.
		\]
		By Minkowski's and  H\"older's inequalities,
		\begin{align*}
			&|Q|^{1/t-\frac1n\sum_{i=1}^n1/p_i}
			\left\|
			\left(
			\sum_{k=1}^{\infty}
			|\mathcal M_{\mathcal D_{\vec a}}
			(f_k\chi_{(3Q)^c})|^u
			\right)^{1/u}\chi_Q
			\right\|_{L^{\vec p}}
			\\
			&\quad\le
			|Q|^{1/t}
			\sum_{m=1}^{\infty}
			\sum_{R\in\mathcal R_m(Q)}
			\frac1{|R|}
			\int_R
			\left(\sum_{k=1}^{\infty}|f_k(y)|^u\right)^{1/u}
			\,\mathrm d y
			\\
			&\quad\le
			\sum_{m=1}^{\infty}2^{-mn/t}
			\sum_{R\in\mathcal R_m(Q)}
			|R|^{1/t-\frac1n\sum_{i=1}^n1/p_i}
			\left\|
			\left(\sum_{k=1}^{\infty}|f_k|^u\right)^{1/u}
			\chi_R
			\right\|_{L^{\vec p}}.
	\end{align*}
	
	For fixed $Q,m$, the cardinality of
		$\mathcal R_m(Q)$ is bounded by a constant depending
		only on $n$. For fixed $R,m$, the number of cubes
		$Q$ satisfying $R\in\mathcal R_m(Q)$ is at most
		$C_n2^{mn}$. Thus, by Minkowski's inequality and
		the same counting argument as in
		Theorem \ref{HL M}, we obtain
		\begin{align*}
			&\left(
			\sum_{Q\in\mathcal D_{\vec a}}
			|Q|^{r/t-\frac rn\sum_{i=1}^n1/p_i}
			\left\|
			\left(
			\sum_{k=1}^{\infty}
			|\mathcal M_{\mathcal D_{\vec a}}
			(f_k\chi_{(3Q)^c})|^u
			\right)^{1/u}\chi_Q
			\right\|_{L^{\vec p}}^r
			\right)^{1/r}
			\\
			&\quad\lesssim
			\sum_{m=1}^{\infty}
			2^{-mn/t}2^{mn/r}
			\|\{f_k\}_{k\in\mathbb N}
			\|_{M_{\vec p}^{t,r}(\ell^u)}
			\\
			&\quad\lesssim
			\|\{f_k\}_{k\in\mathbb N}
			\|_{M_{\vec p}^{t,r}(\ell^u)},
		\end{align*}
		since $t<r$.
		Combining the local and far estimates and using
		(\ref{M le Mdya}) and Lemma \ref{D equivalence}
		completes the proof.
\end{proof}

Moreover, to consider the application of the wavelet characterization, we need the following vector-valued inequality.

\begin{theorem} \label{mixed vector HL mixed BM}
	Let $ 1< \vec p  <   \infty $ and $1<u_1, u_2 < \infty $. Let    $1 < n / ( \sum_{i=1}^n  1/p_i )   < t <r <\infty $ or $ 1< n / ( \sum_{i=1}^n  1/p_i )  \le t < r=\infty  $.
	Then for all $\{ f_{k_1, k_2}  \}_{k_1, k_2 \in \mathbb N} \subset L^0 $, we have
	\begin{equation*}
		\left\| \left( \sum_{k_2 =1}^\infty \left( \sum_{k_1 =1}^\infty  |\M f_{k_1, k_2}|^{u_1} \right) ^{u_2 /u_1 } \right)^{1/u_2}  \right\|_{ M_{\vec p}^{t,r} }   \lesssim  \left\| \left( \sum_{k_2 =1}^\infty \left( \sum_{k_1 =1}^\infty  |f_{k_1, k_2}|^{u_1} \right) ^{u_2 /u_1 } \right)^{1/u_2}  \right\|_{ M_{\vec p}^{t,r} }  .
	\end{equation*}
\end{theorem}
\begin{proof}
The case $r=\infty$ follows from
	\cite[Proposition 2.15]{N24}.
	Thus we assume $r<\infty$.
	By (\ref{M le Mdya}) and Lemma \ref{D equivalence},
	it suffices to estimate
	$\mathcal M_{\mathcal D_{\vec a}}$ for each
	$\vec a\in\{0,1,2\}^n$.
	For $Q\in\mathcal D_{\vec a}$, split
	\[
	f_{k_1,k_2}
	=f_{k_1,k_2}\chi_{3Q}
	+f_{k_1,k_2}\chi_{(3Q)^c}.
	\]
	By \cite[A.2]{N24},
	\begin{align*}
		\left\|
		\left(
		\sum_{k_2=1}^{\infty}
		\left(
		\sum_{k_1=1}^{\infty}
		|\mathcal M_{\mathcal D_{\vec a}}
		(f_{k_1,k_2}\chi_{3Q})|^{u_1}
		\right)^{u_2/u_1}
		\right)^{1/u_2}\chi_Q
		\right\|_{L^{\vec p}}
	\lesssim
		\left\|
		\left(
		\sum_{k_2=1}^{\infty}
		\left(
		\sum_{k_1=1}^{\infty}|f_{k_1,k_2}|^{u_1}
		\right)^{u_2/u_1}
		\right)^{1/u_2}\chi_{3Q}
		\right\|_{L^{\vec p}}.
	\end{align*}
	The first part is estimated by the same covering
	and counting argument for $3Q$ as in Theorem \ref{HL seq mix BM}.

For $m\ge1$, let
	\[
	\mathcal R_m(Q):=
	\{R\in\mathcal D_{\vec a}:
	\ell(R)=2^m\ell(Q),\ R\cap Q\neq\emptyset\}.
	\]
	For $x\in Q$, as in Theorem \ref{HL seq mix BM},
	\[
	\mathcal M_{\mathcal D_{\vec a}}
	(f_{k_1,k_2}\chi_{(3Q)^c})(x)
	\le
	\sum_{m=1}^{\infty}
	\sum_{R\in\mathcal R_m(Q)}
	\frac1{|R|}
	\int_R|f_{k_1,k_2}(y)|\,\mathrm d y.
	\]
	Using Minkowski's inequality twice, we obtain
	\begin{align*}
		&\left(
		\sum_{k_2=1}^{\infty}
		\left(
		\sum_{k_1=1}^{\infty}
		|\mathcal M_{\mathcal D_{\vec a}}
		(f_{k_1,k_2}\chi_{(3Q)^c})(x)|^{u_1}
		\right)^{u_2/u_1}
		\right)^{1/u_2}
		\\
		&\quad\le
		\sum_{m=1}^{\infty}
		\sum_{R\in\mathcal R_m(Q)}
		\frac1{|R|}
		\int_R
		\left(
		\sum_{k_2=1}^{\infty}
		\left(
		\sum_{k_1=1}^{\infty}|f_{k_1,k_2}(y)|^{u_1}
		\right)^{u_2/u_1}
		\right)^{1/u_2}\,\mathrm d y.
\end{align*}
By the  H\"older inequality and
	$|R|=2^{mn}|Q|$ for $R\in\mathcal R_m(Q)$,
	\begin{align*}
		&|Q|^{1/t-\frac1n\sum_{i=1}^n1/p_i}
		\left\|
		\left(
		\sum_{k_2=1}^{\infty}
		\left(
		\sum_{k_1=1}^{\infty}
		|\mathcal M_{\mathcal D_{\vec a}}
		(f_{k_1,k_2}\chi_{(3Q)^c})|^{u_1}
		\right)^{u_2/u_1}
		\right)^{1/u_2}\chi_Q
		\right\|_{L^{\vec p}}
		\\
		&\quad\le
		\sum_{m=1}^{\infty}2^{-mn/t}
		\sum_{R\in\mathcal R_m(Q)}
		|R|^{1/t-\frac1n\sum_{i=1}^n1/p_i}
		\left\|
		\left(
		\sum_{k_2=1}^{\infty}
		\left(
		\sum_{k_1=1}^{\infty}|f_{k_1,k_2}|^{u_1}
		\right)^{u_2/u_1}
		\right)^{1/u_2}\chi_R
		\right\|_{L^{\vec p}}.
	\end{align*}
	Using  the same counting argument for
	$\mathcal R_m(Q)$ as in Theorem \ref{HL seq mix BM} and
	Minkowski's inequality, we obtain
	\begin{align*}
		&\left(
		\sum_{Q\in\mathcal D_{\vec a}}
		|Q|^{r/t-\frac rn\sum_{i=1}^n1/p_i}
		\left\|
		\left(
		\sum_{k_2=1}^{\infty}
		\left(
		\sum_{k_1=1}^{\infty}
		|\mathcal M_{\mathcal D_{\vec a}}
		(f_{k_1,k_2}\chi_{(3Q)^c})|^{u_1}
		\right)^{u_2/u_1}
		\right)^{1/u_2}\chi_Q
		\right\|_{L^{\vec p}}^r
		\right)^{1/r}
		\\
		&\quad\lesssim
		\sum_{m=1}^{\infty}2^{-mn/t}2^{mn/r}
		\left\|
		\left(
		\sum_{k_2=1}^{\infty}
		\left(
		\sum_{k_1=1}^{\infty}|f_{k_1,k_2}|^{u_1}
		\right)^{u_2/u_1}
		\right)^{1/u_2}
		\right\|_{M_{\vec p}^{t,r}}
		\\
		&\quad\lesssim
		\left\|
		\left(
		\sum_{k_2=1}^{\infty}
		\left(
		\sum_{k_1=1}^{\infty}|f_{k_1,k_2}|^{u_1}
		\right)^{u_2/u_1}
		\right)^{1/u_2}
		\right\|_{M_{\vec p}^{t,r}} .
	\end{align*}
	Combining the above estimates and using
	(\ref{M le Mdya}) and Lemma \ref{D equivalence}
	completes the proof.
\end{proof}

\begin{lemma}[Proposition 6.4, \cite{N19}]
	\label{Mit vector}
	Let $ 1 < \vec p <\infty$  and $\omega _k  \in A_{p_k} (\mathbb R)$ for $k =1 , \ldots, n$. Let $1<u \le \infty$. Then for every sequence $\{f_j\}_{j\in\mathbb N}\subset L^0$,
	\begin{equation*}
		\left\|  \left(  \sum_{j=1}^\infty [ \Mit f_j ]^u  \right) ^{1/u}  \cdot \bigotimes_{k=1}^n \omega_k ^{1/p_k} \right\|_{L^{\vec p} } \lesssim \left\|  \left(  \sum_{j=1}^\infty |f_j |^u  \right) ^{1/u}  \cdot \bigotimes_{k=1}^n \omega_k ^{1/p_k} \right\|_{L^{\vec p} }
	\end{equation*}
with the usual modification when $u=\infty$.
\end{lemma}

\begin{theorem}
	Let $ 1< \vec p < \infty $. Let    $ 1 < n / ( \sum_{i=1}^n  1/p_i )   < t <r <\infty $ or $ 1 < n / ( \sum_{i=1}^n  1/p_i )  \le t < r=\infty  $. 
	Furthermore, let  
	\begin{equation*}
		\left( \sum_{i=1 } ^n \frac{1}{p_i} -\frac{n}{t}  +  \frac{n}{r} \right) <  \frac{1}{\max \vec p}.
	\end{equation*}
	Let $1 <u \le \infty $.
	Then for $ \{ f_k\}_{k\in \mathbb N} \in  M_{\vec p}^{t,r} (\ell^u)$,
	\begin{equation*}
		\left(\sum_{Q \in \D} |Q|^{ \frac{r}{t}  - \frac{r}{n} \sum_{i=1 } ^n \frac{1}{p_i}  }  \left\| \left( \sum_{k=1}^\infty  |\Mit f_k|^u \right)^{1/u}  \chi_Q \right\|_{L^{\vec p}  } ^r  \right)^{1/r}   \lesssim \| \{ f_k\}_{k\in \mathbb N} \|_{ M_{\vec p}^{t,r} (\ell^u)  } .
	\end{equation*}
\end{theorem}
\begin{proof}
	Let $\eta \in (0,1)  $  such that
	\begin{equation*}
		0 \le   \left( \sum_{i=1 } ^n \frac{1}{p_i} -\frac{n}{t}  +  \frac{n}{r} \right)  <\eta  < \frac{1}{\max \vec p}<1 .
	\end{equation*}
	By Proposition \ref{mixed BM equ norm weight}, it suffices to show 
	\begin{equation*}
		\left\| \left( \sum_{k=1}^\infty  |\Mit f_k|^u \right)^{1/u}   (\Mit \chi_Q )^\eta  \right\|_{L^{\vec p}  }  \lesssim 	\left\| \left( \sum_{k=1}^\infty  | f_k|^u \right)^{1/u}   (\Mit \chi_Q )^\eta  \right\|_{L^{\vec p}  } .
	\end{equation*}
	Let $Q = I_1 \times \dots \times I_n $. Then 
	\begin{equation*}
		( \Mit \chi_Q) ^{\eta }  = \left(  \bigotimes_{i=1}^n \M _{(i)} \chi_{I_i} \right)^{\eta }  =  \bigotimes_{i=1}^n ( \M_{(i)} \chi_{I_i} )^{\eta } .
	\end{equation*}
	From \cite[Theorem 288]{SDH20}, $ ( \M _{(i)} \chi_{I_i} )^{\eta  p_i} $  belongs to $A_1(\mathbb R)$ if  $\eta > 0$ such that
	$
	0 < \eta  p_i < 1 .
	$
	Hence $ ( \M _{(i)} \chi_{I_i} )^{\eta  p_i} \in A_1(\mathbb R) \subset A_{p_i}(\mathbb R) $ for all $p_i$. By Lemma \ref{Mit vector}, we obtain
	\begin{align*}
		\left\| \left( \sum_{k=1}^\infty  |\Mit f_k|^u \right)^{1/u}   (\Mit \chi_Q )^\eta  \right\|_{L^{\vec p}  }	&  = \left\| \left( \sum_{k=1}^\infty  |\Mit f_k|^u \right)^{1/u}   \bigotimes_{i=1}^n (\M_{(i)} \chi_{I_i } ) ^\eta \right\|_{L^{\vec p}  }	  \\
		& \lesssim \left\| \left( \sum_{k=1}^\infty  | f_k|^u \right)^{1/u}   \bigotimes_{i=1}^n (\M_{(i)} \chi_{I_i } ) ^\eta \right\|_{L^{\vec p}  }	  \\
		&=  \left\| \left( \sum_{k=1}^\infty  | f_k|^u \right)^{1/u}   	( \Mit \chi_Q) ^{\eta }  \right\|_{L^{\vec p}  }	 .
	\end{align*}
	We finish the proof.
\end{proof}

\subsection{The Hardy-Littlewood maximal operator on  block spaces}
We first give the definition of vector valued block space, which is similar with Definition \ref{def mix block space}.

\begin{definition}
	Let $ 1 \le \vec p <\infty$ and $1 < u' \le \infty $. Let   $n / ( \sum_{i=1}^n  1/p_{i})  \le t \le r \le \infty $. A measurable vector valued function $\vec b =\{b_k \}_{k\in \mathbb N_0}$ is said to be a $\ell^{u' }$ valued
	$(\vec{p}^{\,\prime}, t')$-block if there exists a cube $Q$ that supports $\vec b$ such that 
	\begin{equation*}
		\| \vec b \|_{L^{\vec{p}^{\,\prime} } (\ell^{u'}) } : =  \left\| \left( \sum_{m =0 } ^\infty | b_m| ^{u'} \right)^{1/u'} \right\|_{L^{\vec{p}^{\,\prime} } }  \le |Q| ^{  1/t - \frac{1}{n}  ( \sum_{i=1}^n  1/p_{i})   }  .
	\end{equation*}
	If we need to indicate $Q$, we  say that $ \vec b$ is a  $\ell^{u' }$ valued $(\vec{p}^{\,\prime},t')$-block supported on $Q$.
	
	The function space $\mathcal{H}_{\vec p \, ^\prime}^{t',r'} (\ell^ {u'}) $
	is the set of  all measurable functions $\vec f = \{f_m \}_{m \in \mathbb N_0}$ such that $\vec f$ is realized
	as the sum
	\begin{equation}\label{eq: mix block vector f}
		\vec	f = \sum_{(j,k)\in\mathbb{Z}^{n+1}}\lambda_{j,k} \vec b_{j,k}
	\end{equation}
	with some $\lambda=\{\lambda_{j,k}\}_{(j,k)\in\mathbb{Z}^{n+1}}\in\ell^{r'}(\mathbb{Z}^{n+1})$
	and $\vec b_{j,k}$ is a $\ell^{u'}$ valued $(\vec{p}^{\,\prime},t')$-block supported on  $Q_{j,k}$ where (\ref{eq: mix block vector f}) converges almost everywhere on $\rn$. The norm of $\mathcal{H}_{\vec p \, ^\prime}^{t',r'}(\ell^ {u'}) $
	is defined by
	\[
	\| \vec f \|_{\mathcal{H}_{\vec p \, ^\prime}^{t',r'} (\ell^ {u'})} :=\inf_{\lambda}\|\lambda\|_{\ell^{r'}},
	\]
	where the infimum is taken over all admissible sequences $\lambda$
	such that (\ref{eq: mix block vector f}) holds.
\end{definition}

\begin{lemma}[Theorem 2.18, \cite{N24}] \label{HL mix block r = infty}
	Let $ 1< \vec p < \infty $ and $1<u' \le \infty $. Let    $ 1< n / ( \sum_{i=1}^n  1/p_i )  \le t < r=\infty  $.
	Then for all $f  \in \mathcal{H}_{\vec p \, ^\prime}^{t',r'}  (\ell^{u'})  $, we have
	\begin{equation*}
		\left\|  \left(   \sum_{j=0}^\infty \M (f_j) ^{u'}\right)^{1/ u'}   \right\|_{\mathcal{H}_{\vec p \, ^\prime}^{t',r'}  } \lesssim 	\left\|  \left(   \sum_{j=0}^\infty |f_j| ^{u'}\right)^{1/ u'}   \right\|_{\mathcal{H}_{\vec p \, ^\prime}^{t',r'}  } .
	\end{equation*}
\end{lemma}

\begin{theorem} \label{HL mix block r le infty}
	Let $ 1< \vec p <\infty $ and $1<u' \le \infty $. Let    $1 < n / ( \sum_{i=1}^n  1/p_i )   < t <r <\infty $ or $ 1< n / ( \sum_{i=1}^n  1/p_i )  \le t < r=\infty  $.
	Then for all $\{f_j \}_{j\in \mathbb N_0}  \in \mathcal{H}_{\vec p \, ^\prime}^{t',r'}  (\ell^{u'})  $, we have
	\begin{equation*}
		\left\|  \left(   \sum_{j=0}^\infty \M (f_j) ^{u'}\right)^{1/ u'}   \right\|_{\mathcal{H}_{\vec p \, ^\prime}^{t',r'}  } \lesssim 	\left\|  \left(   \sum_{j=0}^\infty |f_j| ^{u'}\right)^{1/ u'}   \right\|_{\mathcal{H}_{\vec p \, ^\prime}^{t',r'}  } .
	\end{equation*}
\end{theorem}

\begin{remark} \label{remark HL block}
	In \cite[Theorem 4.3]{ST09}, Sawano and Tanaka proved the Hardy-Littlewood  maximal operator $\M$ on vector valued on block spaces $ \mathcal{H}_{p'}^{t',1} (\ell^{u'})$; see also \cite[Theorem 2.12]{IST15}. 
	Since Lemma \ref{HL mix block r = infty}, 
	we only need to prove the case  $1 < n / ( \sum_{i=1}^n  1/p_i )   < t <r <\infty $.  We use the idea from \cite[Section 6.2 Besov-Bourgain-Morrey spaces]{ZYY242}. Before proving  Theorem \ref{HL mix block r le infty}, we give some preparation.
\end{remark}

\begin{lemma} \label{norm equ}
	Let $ 0 < \vec p \le \infty $. Let  $   0 < n / ( \sum_{i=1}^n  1/p_{i})     < t <r <\infty $ or $ 0< n / ( \sum_{i=1}^n  1/p_{i})   \le t < r=\infty  $.
	Then $f \in M_{\vec p}^{t,r} $ if and only if $f \in L_{\mathrm{loc}}^{\vec p} $ and
	\begin{equation*}
		\|f\|'_{ M_{\vec p}^{t,r} } := \left(  \int_0^\infty \int_\rn \left( |B(y,\alpha)| ^{1/t - \frac{\sum_{i=1}^n \frac{1}{ p_i} }{n}  -1/r}  \left\| f \chi_{B(y,\alpha)  } \right\|_{L^{\vec p} }  \right) ^r \d y \frac{\d \alpha }{\alpha}   \right)^{1/r} <\infty,
	\end{equation*}
	with the usual modifications made when $r =\infty$. Moreover,
	\begin{equation*}
		\|f \|_{ M_{\vec  p}^{t,r}}  \approx \|f\|'_{ M_{\vec p}^{t,r} } .
	\end{equation*}
\end{lemma}
\begin{proof}
	Repeating  \cite[Theorem 2.9]{ZSTYY23} and replacing  $ \| \cdot \|_{L^p}$ by $\| \cdot \|_{L^{\vec p}}$, we obtain the result. 
\end{proof}

\begin{remark} 
	Let $1<u\le \infty $. 	Let $ 1< \vec p < \infty $ and    $1 < n / ( \sum_{i=1}^n  1/p_i )   < t <r <\infty $.  From Theorem \ref{HL seq mix BM}, we know that  the vector  Hardy-Littlewood maximal  operator $\M$ is bounded on  $ M_{\vec p}^{t,r} (\ell^{u})$. By Lemma \ref{norm equ}, we obtain
	\begin{equation*}
		\|\{ \M  f_k \}_{k\in\mathbb N} \|_{ M_{\vec p}^{t,r} (\ell^u) } ' \lesssim \| \{   f_k \}_{k\in\mathbb N}\|_{ M_{\vec p}^{t,r} (\ell^u) } '.
	\end{equation*}		
\end{remark}
\begin{lemma}\label{HL each j}
	Let $1< \vec p <\infty$,
		$1<r<\infty$, $0<t<\infty$, and $1<u\le\infty$.
		For every $v\in\mathbb Z$, we have
		\begin{align*}
			&\left(
			\sum_{m\in\mathbb Z^n}
			|Q_{v,m}|^{r/t-\frac rn\sum_{i=1}^n1/p_i}
			\left\|
			\left(
			\sum_{\ell=1}^{\infty}|\mathcal M f_\ell|^u
			\right)^{1/u}\chi_{Q_{v,m}}
			\right\|_{L^{\vec p}}^r
			\right)^{1/r}
			\\
			&\qquad\lesssim
			\left(
			\sum_{m\in\mathbb Z^n}
			|Q_{v,m}|^{r/t-\frac rn\sum_{i=1}^n1/p_i}
			\left\|
			\left(
			\sum_{\ell=1}^{\infty}|f_\ell|^u
			\right)^{1/u}\chi_{Q_{v,m}}
			\right\|_{L^{\vec p}}^r
			\right)^{1/r},
		\end{align*}
		with a constant independent of $v$.
\end{lemma}

\begin{proof}
	We first assume $1<u<\infty$.
		Since $|Q_{v,m}|=2^{-vn}$ is independent of $m$,
		the factor
		$|Q_{v,m}|^{1/t-\frac1n\sum_{i=1}^n1/p_i}$
		cancels from both sides.
		For each $m\in\mathbb Z^n$, split
		\[
		f_\ell=f_\ell\chi_{3Q_{v,m}}
		+f_\ell\chi_{(3Q_{v,m})^c}.
		\]
		By \cite[Theorem 1.7]{N19},
		\[
		\left\|
		\left(
		\sum_{\ell=1}^{\infty}
		|\mathcal M(f_\ell\chi_{3Q_{v,m}})|^u
		\right)^{1/u}\chi_{Q_{v,m}}
		\right\|_{L^{\vec p}}
		\lesssim
		\left\|
		\left(
		\sum_{\ell=1}^{\infty}|f_\ell|^u
		\right)^{1/u}\chi_{3Q_{v,m}}
		\right\|_{L^{\vec p}}.
		\]
		Using the same covering and counting argument
		for $3Q_{v,m}$ as in Theorem \ref{HL seq mix BM},
		we obtain
		\begin{align*}
			&\left(
			\sum_{m\in\mathbb Z^n}
			\left\|
			\left(
			\sum_{\ell=1}^{\infty}
			|\mathcal M(f_\ell\chi_{3Q_{v,m}})|^u
			\right)^{1/u}\chi_{Q_{v,m}}
			\right\|_{L^{\vec p}}^r
			\right)^{1/r}
			\lesssim
			\left(
			\sum_{m\in\mathbb Z^n}
			\left\|
			\left(
			\sum_{\ell=1}^{\infty}|f_\ell|^u
			\right)^{1/u}\chi_{Q_{v,m}}
			\right\|_{L^{\vec p}}^r
			\right)^{1/r}.
	\end{align*}
	
	Fix $x,z\in Q_{v,m}$ and let $I$ be a cube
		containing $x$.
		If $I\subset3Q_{v,m}$, then
		\[
		\int_I|f_\ell(y)\chi_{(3Q_{v,m})^c}(y)|
		\,\mathrm d y=0.
		\]
		Otherwise, $\ell(I)\ge2^{-v}$.
		Choose a cube $J$ containing $z$ and all cubes
		$Q_{v,k}$ that intersect $I$, with
		$\ell(J)\le5\ell(I)$.
		Then
		\begin{align*}
			\frac1{|I|}
			\int_I|f_\ell(y)\chi_{(3Q_{v,m})^c}(y)|
			\,\mathrm d y
			&\le
			\frac1{|I|}
			\sum_{\substack{k\in\mathbb Z^n\\
					Q_{v,k}\cap I\ne\varnothing}}
			\int_{Q_{v,k}}|f_\ell(y)|\,\mathrm d y
			\\
			&\le
			\frac{5^n}{|J|}
			\int_J
			\sum_{k\in\mathbb Z^n}
			\left(
			\frac1{|Q_{v,k}|}
			\int_{Q_{v,k}}|f_\ell(y)|\,\mathrm d y
			\right)\chi_{Q_{v,k}}(\xi)\,\mathrm d\xi.
		\end{align*}
		Taking the supremum over $I$, we obtain,
		for every $x,z\in Q_{v,m}$,
		\[
		\mathcal M(f_\ell\chi_{(3Q_{v,m})^c})(x)
		\lesssim
		\mathcal M\left(
		\sum_{k\in\mathbb Z^n}
		\left(
		\frac1{|Q_{v,k}|}
		\int_{Q_{v,k}}|f_\ell(y)|\,\mathrm d y
		\right)\chi_{Q_{v,k}}
		\right)(z).
		\]	
	Taking the $\ell^u$-norm and then the
		$L^r$-average over $z\in Q_{v,m}$, we obtain
		\begin{align*}
			&\left(
			\sum_{m\in\mathbb Z^n}
			\left\|
			\left(
			\sum_{\ell=1}^{\infty}
			|\mathcal M(f_\ell\chi_{(3Q_{v,m})^c})|^u
			\right)^{1/u}\chi_{Q_{v,m}}
			\right\|_{L^{\vec p}}^r
			\right)^{1/r}
			\\
			&\quad\lesssim
			2^{-v(\sum_{i=1}^n1/p_i-n/r)}
			\left\|
			\left(
			\sum_{\ell=1}^{\infty}
			\left|
			\mathcal M\left(
			\sum_{k\in\mathbb Z^n}
			\left(
			\frac1{|Q_{v,k}|}
			\int_{Q_{v,k}}|f_\ell(y)|\,\mathrm d y
			\right)\chi_{Q_{v,k}}
			\right)
			\right|^u
			\right)^{1/u}
			\right\|_{L^r}
			\\
			&\quad\lesssim
			2^{-v\sum_{i=1}^n1/p_i}
			\left(
			\sum_{k\in\mathbb Z^n}
			\left(
			\sum_{\ell=1}^{\infty}
			\left(
			\frac1{|Q_{v,k}|}
			\int_{Q_{v,k}}|f_\ell(y)|\,\mathrm d y
			\right)^u
			\right)^{r/u}
			\right)^{1/r},
		\end{align*}
		where the last step follows from the
		Fefferman--Stein vector-valued maximal inequality
		on $L^r$ and $|Q_{v,k}|=2^{-vn}$.		
	By Minkowski's and H\"older's inequalities,
		\begin{align*}
			\left(
			\sum_{\ell=1}^{\infty}
			\left(
			\frac1{|Q_{v,k}|}
			\int_{Q_{v,k}}|f_\ell(y)|\,\mathrm d y
			\right)^u
			\right)^{1/u}
			&\le
			\frac1{|Q_{v,k}|}
			\int_{Q_{v,k}}
			\left(
			\sum_{\ell=1}^{\infty}|f_\ell(y)|^u
			\right)^{1/u}\,\mathrm d y
			\\
			&\le
			2^{v\sum_{i=1}^n1/p_i}
			\left\|
			\left(
			\sum_{\ell=1}^{\infty}|f_\ell|^u
			\right)^{1/u}\chi_{Q_{v,k}}
			\right\|_{L^{\vec p}}.
		\end{align*}
		Combining the above estimates proves the result
		for $1<u<\infty$.
		Taking a sequence with only one nonzero component
		also gives the scalar estimate.
		Thus the case $u=\infty$ follows from
		\[
		\sup_{\ell\in\mathbb N}\mathcal M f_\ell
		\le
		\mathcal M\left(
		\sup_{\ell\in\mathbb N}|f_\ell|
		\right).
		\]
	Thus the proof is finished.
\end{proof}

\begin{remark}
	Let $1<u\le\infty$, $1<\vec p<\infty$,
		$1<r<\infty$, and $0<t<\infty$.

	For each $j\in\mathbb Z$, by Lemma \ref{HL each j}
		and a covering argument at the fixed scale $2^{-j}$,
		we obtain
	\begin{align*}
		& 	 \int_\rn \left( |B(y,2^{-j})| ^{1/t -\frac{\sum_{i=1}^n \frac{1}{ p_i} }{n}  -1/r}
		\left\| \left(  \sum_{\ell=1}^\infty ( \M f_\ell ) ^u \right)^{1/u} \chi_{B(y, 2^{-j})  } \right\|_{L^{\vec p} }\right)^{r} \d y
		\\
		& \lesssim 	 \int_\rn \left( |B(y,2^{-j})| ^{1/t -\frac{\sum_{i=1}^n \frac{1}{ p_i} }{n}  -1/r}
		\left\| \left(  \sum_{\ell=1}^\infty | f_\ell | ^u \right)^{1/u} \chi_{B(y, 2^{-j})  } \right\|_{L^{\vec p} }\right)^{r} \d y  .
	\end{align*}
\end{remark}

\begin{definition}\label{def vec skice}
	
	Let $1< u' <\infty$ and $1 <\vec p <\infty$.
	Let    $1 < n / ( \sum_{i=1}^n  1/p_i )   < t <r <\infty $ and $j \in \mathbb Z$.	 
	The slice space $  (\mathcal E_{\vec p \, ^\prime}^{t',r'}   )_j  (\ell^{u'}) $ is defined to be the set of all  $ \{ f_w \}_{w\in \mathbb N} \in L_{\operatorname{loc} } ^{\vec p \, ^\prime} (\ell^{u'})$  such that 
	\begin{align*}
	 \|\{ f_w \}_{w\in \mathbb N} \|_{  (\mathcal E_{\vec p \, ^\prime}^{t',r'}   )_j  (\ell^{u'} ) } 
	 := \left( \sum_{k\in \mathbb Z^n } \left( |Q_{j,k}|^{1/t' -\frac{\sum_{i=1}^n \frac{1}{ p_i '} }{n} }   \left\| \left(\sum_{w=1}^\infty |f_w|^{u'}  \right)^{1/u'}  \chi_{ Q_{j,k} } \right \|_{ L^{\vec p \, ^\prime}  }   \right) ^{r'}  \right) ^{1/r'} <\infty.
	\end{align*}
\end{definition}

Applying Lemma \ref{HL each j} with
	$(\vec p,t,r,u)$ replaced by
	$(\vec p',t',r',u')$, we obtain that
	$\mathcal M$ is bounded on
	$(\mathcal E_{\vec p'}^{t',r'})_j(\ell^{u'})$,
	with a constant independent of $j$.

In the following lemma, we establish a characterization of  slice space  $  (\mathcal E_{\vec p \, ^\prime}^{t',r'}   )_j  (\ell^{u'}) $.

\begin{lemma} \label{char vec slice}
	Let $1< u' <\infty$ and $1 <\vec p <\infty$.
	Let    $1 < n / ( \sum_{i=1}^n  1/p_i )   < t <r <\infty $ and $j \in \mathbb Z$.  Then $\vec f = \{ f_w\}_{w\in \mathbb N} \in (\mathcal E_{\vec p \, ^\prime}^{t',r'}   )_j  (\ell^{u'})  $ if and only if there exist a sequence $  \{ \lambda_{j,k}\}_{k\in \mathbb Z^n}  \subset \mathbb R$ satisfying
	\begin{equation*}
		\left(  \sum_{ k\in \mathbb Z^n} |\lambda_{j,k}|^{r'} \right)^{1/r'} <\infty
	\end{equation*}
	and a sequence $\{\vec  b_{j,k} \}_{k\in \mathbb Z^n}$ of vector functions on $\rn$ satisfying, for each $k\in \mathbb Z^n$, supp $\vec b_{j,k}\subset Q_{j,k}$  and 
	\begin{equation} \label{b jk = |Q|}
		\left\|  \| \vec b_{j,k}\|_{\ell^{u'}}  \right\|_{ L^{ \vec p \, ^\prime} (Q_{j,k}) }  = |Q_{j,k}|^{1/t - \frac{\sum_{i=1}^n \frac{1}{ p_i} }{n} }
	\end{equation}
	such that $\vec f = \sum_{k \in \mathbb Z^n} \lambda_{j,k} \vec  b_{j,k} $ almost everywhere on $\rn$; moreover, for such $\vec f$,
	\begin{equation*}
		\| \vec  f\|_{ (\mathcal E_{\vec p \, ^\prime}^{t',r'}  )_j (\ell^{ u ' }) }  = 	\left(  \sum_{ k \in \mathbb Z^n} |\lambda_{j,k}|^{r'} \right)^{1/r'}  .
	\end{equation*}
\end{lemma}
\begin{proof}
	We first show the necessity. Let $\vec f = \{ f_w\}_{w\in \mathbb N}  \in  (\mathcal E_{\vec p \, ^\prime}^{t',r'}   )_j  (\ell^{u'}) $. For each $k \in \mathbb Z^n$, when $  \left\|  \|\vec f\|_{\ell^{u'}} \chi_{Q_{j,k} } \right\|_{ L^{\vec p \, ^\prime}  } >0 $, let $ \lambda_{j,k} :=  |Q_{j,k}|^{ \frac{\sum_{i=1}^n \frac{1}{ p_i} }{n} -1/t }   \| \chi_{Q_{j,k}} \vec f \|_{L^{\vec p \, ^\prime}(\ell^{u'})}  $,   $\vec  b_{j,k} := \lambda_{j,k} ^{-1} \vec f \chi_{Q_{j,k} }$ and, when $  \left\|  \|\vec f\|_{\ell^{u'}} \chi_{Q_{j,k} } \right\|_{ L^{\vec p \, ^\prime}  } = 0 $, let $ \lambda_{j,k} := 0$ and $\vec  b_{j,k} : = |Q_{j,k}|^{ -1/t'} \chi_{ Q_{j,k} }  \vec a $ where $\vec a = (1, 0, 0, \ldots) \in \ell^{u'} $ with $\|\vec a\|_{ \ell^{u'}} =1$. Then by the Definition \ref{def vec skice}, it is easy to show that supp $\vec b_{j,k}\subset Q_{j,k}$, (\ref{b jk = |Q|}) and 
	\begin{equation*}
		\left(  \sum_{k\in \mathbb Z^n} |\lambda_{j,k}|^{r'} \right)^{1/r'}  = 	\|\vec f\|_{  (\mathcal E_{\vec p \, ^\prime}^{t',r'}   )_j  (\ell^{u'})  }  	  <\infty .
	\end{equation*} 	
	Next we show the sufficiency. Assume that there exist a sequence $  \{ \lambda_{j,k}\}_{k\in \mathbb Z^n}  \in \ell^{r'}$ and a sequence $\{\vec  b_{j,k} \}_{k\in \mathbb Z^n}$ satisfying supp $\vec b_{j,k} \subset Q_{j,k}$ and  (\ref{b jk = |Q|}) such that $\vec f =  \sum_{k \in \mathbb Z^n} \lambda_{j,k} \vec b_{j,k}$  almost everywhere on $\rn$. Note that $ \{ Q_{j,k}\}_{k\in \mathbb Z^n}$  are disjoint. Then 
	\begin{align*}
		\|\vec  f\|_{  (\mathcal E_{\vec p \, ^\prime}^{t',r'}   )_j  (\ell^{u'})  }  = \left( \sum_{k\in \mathbb Z^n } \left( |Q_{j,k}|^{\frac{\sum_{i=1}^n \frac{1}{ p_i} }{n}  - 1/t }   \| \lambda_{j,k}\vec  b_{j,k}\|_{ L^{\vec p \, ^\prime} (\ell^{u'}) }   \right) ^{r'}  \right) ^{1/r'}   = 	\left(  \sum_{k\in \mathbb Z^n} |\lambda_{j,k}|^{r'} \right)^{1/r'}  <\infty.
	\end{align*}
	Hence $\vec f \in (\mathcal E_{\vec p \, ^\prime}^{t',r'}   )_j  (\ell^{u'})$. Thus we finish the proof.
\end{proof}
Using the characterization of vector valued slice space $(\mathcal E_{\vec p \, ^\prime}^{t',r'}   )_j  (\ell^{u'})$, we obtain the following characterization of block spaces.
\begin{lemma}\label{char block}
	Let $1< u' <\infty$ and $1 <\vec p <\infty$.
	Let    $1 < n / ( \sum_{i=1}^n  1/p_i )   < t <r <\infty $.  Then $\vec f \in \mathcal{H}_{\vec p \, ^\prime}^{t',r'} (\ell^{u'})$ if and only if 
	\begin{align*}
		\| \vec f\|^{\blacklozenge}_{\mathcal{H}_{\vec p \, ^\prime}^{t',r'} (\ell^{u'}) } : = \inf\left\{ \left( \sum_{j\in \mathbb Z} \|\vec  f_j\|_{  (\mathcal E_{\vec p \, ^\prime}^{t',r'}   )_j  (\ell^{u'})  } ^{r'} \right) ^{1/r' }  :  \vec f = \sum_{j\in \mathbb Z} \vec f_j, \vec f_j \in  (\mathcal E_{\vec p \, ^\prime}^{t',r'}   )_j  (\ell^{u'})  \right\}  <\infty ;
	\end{align*}
	moreover, for such $\vec f$, $ \|\vec  f\|_{ \mathcal{H}_{\vec p \, ^\prime}^{t',r'} (\ell^{u'}) } = 	\| \vec f\|^{\blacklozenge}_{\mathcal{H}_{\vec p \, ^\prime}^{t',r'} (\ell^{u'}) }$.
\end{lemma}
\begin{proof}
	We first show the necessity. Let $\vec f \in \mathcal{H}_{\vec p \, ^\prime}^{t',r'} (\ell^{u'}) $. Then	there exist a  sequence $\lambda=\{\lambda_{j,k}\}_{(j,k)\in\mathbb{Z}^{n+1}}\in\ell^{r'}(\mathbb{Z}^{n+1})$ and a  $\ell^{u'} $-valued sequence $\{\vec  b_{j,k}  \}_{(j,k)\in\mathbb{Z}^{n+1} }  $  of  $(\vec p \, ^\prime,t' )$-block such that 
	\begin{equation*}
		\vec f=\sum_{(j,k)\in\mathbb{Z}^{n+1}}\lambda_{j,k} \vec b_{j,k}  
	\end{equation*}
	almost everywhere on $\rn$ and $ \left( \sum_{ (j,k)\in\mathbb{Z}^{n+1} } |\lambda_{j,k}|^{r'}\right)^{1/r'}  < (1+\epsilon) \|\vec  f\|_{ \mathcal{H}_{\vec p \, ^\prime}^{t',r'} (\ell^{u'})}$.
	For each $j \in \mathbb Z$, let $\vec f_j : = \sum_{k \in \mathbb Z^n}  \lambda_{j,k} \vec  b_{j,k}  $. Then we obtain
	\begin{align*}
		\left(  \sum_{j\in \mathbb Z}	\| \vec f_j \|_{ (\mathcal E_{\vec p \, ^\prime}^{t',r'}   )_j  (\ell^{u'}) }  ^{r'}  \right)^{1/r'} 
		& = \left(  \sum_{j\in \mathbb Z} \left(  \sum_{k\in \mathbb Z^n } \left( |Q_{j,k}|^{\frac{\sum_{i=1}^n \frac{1}{ p_i} }{n}  - 1/t}   \| \lambda_{j,k}\vec  b_{j,k}\|_{ L^{\vec p \, ^\prime} (\ell^{u'}) }   \right) ^{r'}  \right)^{r'/r'}   \right) ^{1/r'} \\
		& \le 	\left(  \sum_{j\in \mathbb Z} \sum_{ k \in \mathbb Z^n} |\lambda_{j,k}|^{r'} \right)^{1/r'} < (1+\epsilon) \| \vec  f\|_{ \mathcal{H}_{\vec p \, ^\prime}^{t',r'} (\ell^{u'}) }.
	\end{align*}
	Letting $\epsilon \to 0^+$, we obtain $  	\| \vec f\|^{\blacklozenge}_{\mathcal{H}_{\vec p \, ^\prime}^{t',r'} (\ell^{u'})} \le  \|\vec  f\|_{ \mathcal{H}_{\vec p \, ^\prime}^{t',r'} (\ell^{u'})} $.
	
	Next we show the sufficiency. Let $\vec f $ be a measurable vector function with $ 	\| \vec f\|^{\blacklozenge}_{\mathcal{H}_{\vec p \, ^\prime}^{t',r'} (\ell^{u'})} <\infty$. Then there exists a sequence $\{\vec f_j \}_{j\in \mathbb Z}$ satisfying that $\vec f_j \in  (\mathcal E_{\vec p \, ^\prime}^{t',r'}   )_j (\ell^{u'})  $ for each $j$ such that $\vec f =\sum_{j\in \mathbb Z} \vec f_j $ almost everywhere on $\rn$ and 
	\begin{equation*}
		\left( \sum_{j\in \mathbb Z} \|\vec  f_j\|_{ (\mathcal E_{\vec p \, ^\prime}^{t',r'}   )_j  (\ell^{u'}) } ^{r'} \right) ^{1/r' } < (1+\epsilon) 	\| \vec f\|^{\blacklozenge}_{\mathcal{H}_{\vec p \, ^\prime}^{t',r'} (\ell^{u'})}.
	\end{equation*}
	By  Lemma \ref{char vec slice}, for each $j\in \mathbb Z$, there exist a sequence 
	$  \{\lambda_{j,k}\}_{k\in\mathbb Z^n}    \in \ell^{r'}$  and 
	a vector valued sequence $ \{\vec b_{j,k}\}_{k\in\mathbb Z^n}   $   satisfying (\ref{b jk = |Q|}) such that $\vec f_j = \sum_{k \in \mathbb Z^n} \lambda_{j,k}\vec  b_{j,k} $ almost everywhere on $\rn$ and 
	\begin{equation*}
		\|\vec  f_j \|_{  (\mathcal E_{\vec p \, ^\prime}^{t',r'}   )_j  (\ell^{u'})  }  = 	\left(  \sum_{ k \in \mathbb Z^n} |\lambda_{j,k}|^{r'} \right)^{1/r'} .
	\end{equation*}
	Hence, $\vec f = \sum_{j\in \mathbb Z}\vec f_j = \sum_{j\in \mathbb Z}\sum_{k \in \mathbb Z^n} \lambda_{j,k}\vec  b_{j,k}  $ almost everywhere on $\rn$ and 
	\begin{equation*}
		\|\vec  f\|_{ \mathcal{H}_{\vec p \, ^\prime}^{t',r'} (\ell^{u'})  }  \le 	\left(  \sum_{j\in \mathbb Z} \sum_{ k \in \mathbb Z^n} |\lambda_{j,k}|^{r'} \right)^{1/r'} < (1+\epsilon) 	\| \vec f\|^{\blacklozenge}_{\mathcal{H}_{\vec p \, ^\prime}^{t',r'} (\ell^{u'}) }.
	\end{equation*}
	Letting $\epsilon \to 0^+$, we obtain $	\| \vec f\|_{ \mathcal{H}_{\vec p \, ^\prime}^{t',r'} (\ell^{u'}) } \le   	\| \vec f\|^{\blacklozenge}_{\mathcal{H}_{\vec p \, ^\prime}^{t',r'} (\ell^{u'}) } $. This finish the proof of sufficiency. Thus we complete the proof of Lemma \ref{char block}.
\end{proof}
Now we are ready to show  Theorem \ref{HL mix block r le infty}.
\begin{proof}[Proof of Theorem \ref{HL mix block r le infty}]
	Just as in Remark \ref{remark HL block}, we only need to show the case  $1<u' \le \infty $ and    $1 < n / ( \sum_{i=1}^n  1/p_i )   < t <r <\infty $.
	
	Subcase $1<u' <\infty $ and $1 < n / ( \sum_{i=1}^n  1/p_i )   < t <r <\infty $.	
	
	Let $0<\epsilon <1$.
	Let $\vec f\in \mathcal{H}_{\vec p \, ^\prime}^{t',r'} (\ell^{u'}) $. From Lemma \ref{char block}, there exists a sequence $\{\vec f_j\}_{j\in \mathbb Z} \subset (\mathcal E_{\vec p \, ^\prime}^{t',r'}   )_j (\ell^{u'}) $ such that $\vec f = \sum_{j\in \mathbb Z} \vec f_j$ almost everywhere on $\rn$ and 
	\begin{equation*}
		\| \vec f\|_{\mathcal{H}_{\vec p \, ^\prime}^{t',r'} (\ell^{u'})  } = 	\| \vec f\|^{\blacklozenge}_{\mathcal{H}_{\vec p \, ^\prime}^{t',r'} (\ell^{u'})  } > (1 - \epsilon) \left( \sum_{j\in \mathbb Z} \| \vec f_j\|_{ (\mathcal E_{\vec p \, ^\prime}^{t',r'}   )_j (\ell^{u'})  } ^{r'} \right) ^{1/r' } .
	\end{equation*}
	By the sublinear of Hardy-Littlewood maximal function and Lemma \ref{HL each j}, we obtain
	\begin{align*}
		\| \M \vec f\|_{\mathcal{H}_{\vec p \, ^\prime}^{t',r'} (\ell^{u'}) }  & \le \left\| \sum_{j\in   \mathbb Z}  \M \vec f _j  \right\|_{ \mathcal{H}_{\vec p \, ^\prime}^{t',r'} (\ell^{u'}) }  = \left\| \sum_{j\in   \mathbb Z}  \M  \vec f _j \right\|^{\blacklozenge}_{ \mathcal{H}_{\vec p \, ^\prime}^{t',r'} (\ell^{u'})   }  \\
		& \le \left( \sum_{j\in   \mathbb Z} \| \M  \vec f _j  \|_{ (\mathcal E_{\vec p \, ^\prime}^{t',r'}   )_j (\ell^{u'}) } ^{r'} \right)^{1/r'}  \lesssim  \left( \sum_{j\in   \mathbb Z} \|  \vec f _j  \|_{ (\mathcal E_{\vec p \, ^\prime}^{t',r'}   )_j (\ell^{u'})} ^{r'} \right)^{1/r'}  \\
		& < \frac{1}{1-\epsilon} 	\| \vec f\|_{\mathcal{H}_{\vec p \, ^\prime}^{t',r'} (\ell^{u'}) }
	\end{align*}
	where the implicit positive constants are independent of both $\vec f$ and $\epsilon$. Letting $\epsilon \to 0^+$, we obtain $	\| \M \vec  f\|_{\mathcal{H}_{\vec p \, ^\prime}^{t',r'} (\ell^{u'})} \lesssim  	\| \vec f\|_{\mathcal{H}_{\vec p \, ^\prime}^{t',r'} (\ell^{u'})} $.
	
	Let $ \vec g =\{ g_j \}_{j =0}^\infty  $ be defined by $	g_0 = f, g_j =0$ for $j \in \mathbb N$. Then from the above, we obtain the scalar version:
	\begin{equation} \label{scalar HL mixed block}
		\| \M   f\|_{\mathcal{H}_{\vec p \, ^\prime}^{t',r'} }	=	\| \M \vec  g\|_{\mathcal{H}_{\vec p \, ^\prime}^{t',r'} (\ell^{u'})} \lesssim  	\| \vec g\|_{\mathcal{H}_{\vec p \, ^\prime}^{t',r'} (\ell^{u'})} = \|   f\|_{\mathcal{H}_{\vec p \, ^\prime}^{t',r'} } .
	\end{equation}
	
	Subcase $u' = \infty $ and $1 < n / ( \sum_{i=1}^n  1/p_i )   < t <r <\infty $.	
	By 	\begin{equation*}
		\sup_{k \in \mathbb N_0} \M f_k  \le  \M  \left( \sup_{k \in \mathbb N_0} | f_k| \right) , 
	\end{equation*} and (\ref{scalar HL mixed block}),
	we obtain 
	\begin{equation*}
		\| \M \vec  f\|_{\mathcal{H}_{\vec p \, ^\prime}^{t',r'} (\ell^{\infty})} \lesssim  	\| \vec f\|_{\mathcal{H}_{\vec p \, ^\prime}^{t',r'} (\ell^{\infty})}. 
	\end{equation*}
	Hence, we finish the proof.	
\end{proof}

\subsection{Fractional integral operator}
For $ 0<\alpha <n $, define the fractional integral operator $I_\alpha$ of order $\alpha$ by
\begin{equation*} 
	I_\alpha f (x) := \int_\rn \frac{f (y)}{|x-y|^{n-\alpha} }  \d y
\end{equation*}
for $f \in L^1_{ \operatorname{loc}} $  as long as the right-hand side makes sense.

From the standard argument for fractional integral operators on mixed Morrey spaces (see \cite[proof of Theorm 1.11]{N19}), we get that the pointwise estimate for non-negative functions $f$   is 
\begin{equation} \label{frac estimate}
	I_\alpha f (x)  \lesssim \| f\|_{ M_{\vec p}^{t,\infty} }^{ t\alpha /n  } \M f (x) ^{ 1 - t\alpha /  n }
\end{equation}
where $1<\vec p<\infty$,
$1<n/(\sum_{i=1}^n1/p_i)\le t<\infty$,
and $0<\alpha<n/t$.

\begin{theorem}
	Let $ 1< \vec p_1 , \vec p_2< \infty $.  
	For $j=1,2$, let   $1 < n / ( \sum_{i=1}^n  1/p_{j,i})   < t_j <r_j <\infty $ or $ 1< n / ( \sum_{i=1}^n  1/p_{j,i})  \le t_j < r_j =\infty  $ where $p_{j,i}$  is $i$-th item of the vector $\vec p_j$.
	Let $ 0<\alpha <n $ and $1/ t_2 = 1 /t_1 - \alpha /n $. Assume that
	$
	t_1 \vec p_2 = t_2 \vec p_1, 
	t_1 r_2 = t_2 r_1. 
	$	
	Then 
	$	\|  I_\alpha  f \|_{M_{\vec p_2 }^{t_2,r_2	}  } \lesssim   \| f\|_{ M_{\vec p_1}^{t_1,r_1} } .$

\end{theorem}

\begin{proof}
	We have a pointwise estimate (\ref{frac estimate}) for  $|f|$. 	
	Using the embedding  $M_{\vec p}^{t,r}  \hookrightarrow  M_{\vec p}^{t,\infty} $
	and the condition  $1/ t_2 = 1 /t_1 - \alpha /n  $, 
	we obtain the  estimate
	\begin{equation*}
		| I_\alpha  f(x) | \le  I_\alpha  [ | f| ](x) \lesssim \| f\|_{ M_{\vec p_1}^{t_1,r_1} }^{ t_1 \alpha /n  } \M f (x) ^{ 1 - t_1 \alpha /  n } .
	\end{equation*} 
	Hence
	\begin{align*}
		\|  I_\alpha  f \|_{M_{\vec p_2 }^{t_2,r_2	}  }  & \lesssim \| f\|_{ M_{\vec p_1}^{t_1,r_1} }^{ t_1 \alpha /n  }  	\|  \M f  ^{ 1 - t_1 \alpha /  n }  \|_{M_{\vec p_2 }^{t_2,r_2	}  }   	= \| f\|_{ M_{\vec p_1}^{t_1,r_1} }^{ t_1 \alpha /n  }  	\|  \M f   \|_{M_{  (1 - t_1 \alpha /  n ) \vec p_2 }^{ (1 - t_1 \alpha /  n) t_2, (1 - t_1 \alpha /  n) r_2	}  } ^{ 1 - t_1 \alpha /  n }  \\
		& 	\lesssim  \| f\|_{ M_{\vec p_1}^{t_1,r_1} }^{ t_1 \alpha /n  }  	\|   f   \|_{M_{  (1 - t_1 \alpha /  n ) \vec p_2 }^{ (1 - t_1 \alpha /  n) t_2, (1 - t_1 \alpha /  n) r_2	}  } ^{ 1 - t_1 \alpha /  n }    =  \| f\|_{ M_{\vec p_1}^{t_1,r_1} } .
	\end{align*}
	We finish the proof.
\end{proof}

\subsection{Singular integral operators}
Let us consider the boundedness of singular integral operators.
\begin{definition}
	A singular integral operator is an $L^2$-bounded linear operator $T$ with a kernel $K(x,y)$ which satisfies the following conditions:
	\begin{itemize}
		\item[(i)] There exist  $\epsilon, c>0$  such that $ |K(x,y) | \le c / |x-y|^n $ for $x \neq y$ and 
		\begin{equation*}
			| K (x,y) - K (z,y) | + | K (y,x) -K (y,z) |  \le c \frac{|x-z| ^\epsilon}{|x-y| ^{ \epsilon+n  } }
		\end{equation*}
		if $ |x-y| \ge 2 |x-z| >0$.
		\item [(ii)] If $f\in L_c^\infty  $, then
		\begin{equation*}
			T f(x) = \int_\rn K (x,y) f(y) \d y ,  \quad  x \notin  \operatorname{supp} (f).
		\end{equation*}
	\end{itemize}
\end{definition}

\begin{theorem}
	Let $1 < u <\infty $
	and $ 1< \vec p <\infty$. Let   $1 < n / ( \sum_{i=1}^n  1/p_{i})   < t <r<\infty $ or $ 1< n / ( \sum_{i=1}^n  1/p_{i})  \le t< r =\infty  $. Let $T$ be a singular integral operator. Then for all $\vec f \in M_{\vec p}^{t,r} (\ell^u)$, we have
	\begin{equation*}
		\| T \vec f\|_{M_{\vec p}^{t,r} (\ell^u) }  \lesssim 	\| \vec f\|_{M_{\vec p}^{t,r} (\ell^u) };
	\end{equation*}
	and for all $g \in M_{\vec p}^{t,r}$, we have
	\begin{equation*}
		\| Tg\|_{M_{\vec p}^{t,r} }  \lesssim 	\| g\|_{M_{\vec p}^{t,r} } .
	\end{equation*}
\end{theorem}

\begin{proof}
	It suffices to show the first inequality.	 	
	For each $k$, we have
		$f_k\in M_{\vec p}^{t,r}\hookrightarrow
		M_{\vec p}^{t,\infty}$.
		Thus $Tf_k$ is defined using the scalar extension
		of $T$ to mixed Morrey spaces
		given in \cite[Theorem 7.2]{N19}.
		We define $T\vec f:=\{Tf_k\}_{k\in\mathbb N}$.
	
	Recall that every Calder\'on--Zygmund operator
		$T$ is bounded on $L^s(w)$ for all $1<s<\infty$
		and $w\in A_s$.	
	By Theorem \ref{associate space}, the associate space of
		$M_{\vec p}^{t,r}$ is $\H_{\vec p \, ^ \prime }^{t',r'}$.
		Theorems \ref{HL M} and \ref{HL mix block r le infty} show that $\mathcal M$
		is bounded on both spaces.
	Applying \cite[Theorem 4.13]{Nie23} with
	$X=M_{\vec p}^{t,r}$, $Y=\ell^u$, and extrapolation
	range $(1,\infty)$, we obtain
		\[
		\|T\vec f\|_{M_{\vec p}^{t,r}(\ell^u)}
		\lesssim
		\|\vec f\|_{M_{\vec p}^{t,r}(\ell^u)}.
		\]
		Thus we finish the proof.
\end{proof}

\begin{theorem}\label{T H block}
	Let $ 1< \vec p <\infty$. Let   $1 < n / ( \sum_{i=1}^n  1/p_{i})   < t <r<\infty $ or $ 1< n / ( \sum_{i=1}^n  1/p_{i})  \le t< r =\infty  $. Let $T$ be a singular integral operator. Then for all $g \in \mathcal{H}_{\vec p \, ^\prime}^{t',r'}$, we have
	\begin{equation*}
		\| Tg\|_{\mathcal{H}_{\vec p \, ^\prime}^{t',r'} }  \lesssim 	\| g\|_{\mathcal{H}_{\vec p \, ^\prime}^{t',r'} } .
	\end{equation*}
\end{theorem}
\begin{proof}
	The associate space of $\mathcal{H}_{\vec p \, ^\prime}^{t',r'}$
		is $M_{\vec p}^{t,r}$.
		By Theorems \ref{HL M} and \ref{HL mix block r le infty},
		the maximal operator $\mathcal M$ is bounded on
		both spaces.
		Since $T$ is bounded on $L^s(w)$ for all
		$1<s<\infty$ and $w\in A_s$, applying
		\cite[Theorem 4.6]{Nie23} with
		$X=\mathcal{H}_{\vec p \, ^\prime}^{t',r'}$ gives
		\[
		\|Tg\|_{\mathcal{H}_{\vec p \, ^\prime}^{t',r'}}
		\lesssim
		\|g\|_{\mathcal{H}_{\vec p \, ^\prime}^{t',r'}}.
		\]

	Thus the proof is complete.
\end{proof}
\begin{theorem}\label{T H vector block}
	Let $1<u ' <\infty$ and $ 1< \vec p <\infty$. Let   $1 < n / ( \sum_{i=1}^n  1/p_{i})   < t <r<\infty $ or $ 1< n / ( \sum_{i=1}^n  1/p_{i})  \le t< r =\infty  $. Let $T$ be a singular integral operator. Then for all $ \vec g=\{ g_j\}_{j =0 }^\infty  \in \mathcal{H}_{\vec p \, ^\prime}^{t',r'} (\ell^ {u'} )$, we have
	\begin{equation*}
		\| T \vec g \|_{\mathcal{H}_{\vec p \, ^\prime}^{t',r'} (\ell^ {u'} ) }  \lesssim 	\|\vec  g\|_{\mathcal{H}_{\vec p \, ^\prime}^{t',r'} (\ell^ {u'} ) } .
	\end{equation*}
\end{theorem}

\begin{proof}

	By the definitions of the scalar and
		vector-valued block spaces,
		\[
		\|\vec g\|_{\mathcal H_{\vec p'}^{t',r'}(\ell^{u'})}
		=
		\left\|
		\|\vec g(\cdot)\|_{\ell^{u'}}
		\right\|_{\mathcal H_{\vec p'}^{t',r'}}.
		\]
		The associate space of
		$\mathcal H_{\vec p'}^{t',r'}$ is $M_{\vec p}^{t,r}$.
		By Theorems \ref{HL M} and
		\ref{HL mix block r le infty},
		$\mathcal M$ is bounded on both spaces.
		Since $T$ is bounded on $L^s(w)$ for all
		$1<s<\infty$ and $w\in A_s$,
		applying \cite[Theorem 4.13]{Nie23} with
		$X=\mathcal H_{\vec p'}^{t',r'}$, $Y=\ell^{u'}$,
		and extrapolation range $(1,\infty)$ gives
		\[
		\|T\vec g\|_{\mathcal H_{\vec p'}^{t',r'}(\ell^{u'})}
		\lesssim
		\|\vec g\|_{\mathcal H_{\vec p'}^{t',r'}(\ell^{u'})}.
		\]
		Thus the proof is complete.
\end{proof}

\section*{Declarations}

\noindent\textbf{Conflict of interest}\quad All authors state no conflict of interest.

\noindent\textbf{Author Contributions}\quad 
All authors developed and discussed the results and contributed to the final
manuscript.

\medskip

\noindent\textbf{Data Availability}\quad Data sharing is not applicable 
to this article as no data sets were generated or analyzed.


\bigskip

\noindent   Tengfei Bai, Pengfei Guo

\medskip

\noindent School of Mathematics and Statistics, Hainan Normal University, Haikou, Hainan 571158,
China

\medskip

\noindent Jingshi Xu (Corresponding author)

\medskip

\noindent School of Mathematics and Computing Science, Guangxi Colleges and Universities Key Laboratory of Data Analysis and Computation, Guilin University of Electronic Technology, Guilin, 541004, China 

\noindent Center for Applied Mathematics of Guangxi (GUET), Guilin, 541004, China

\smallskip

\noindent {\it E-mails}:
\texttt{202311070100007@hainnu.edu.cn} (T. Bai)

\noindent\phantom{{\it E-mails:}}
\texttt{guopf999@163.com} (P. Guo)

\noindent\phantom{{\it E-mails:}}
\texttt{jingshixu@126.com} (J. Xu)

\end{document}